\documentclass[12pt]{amsart}
\usepackage{fullpage}

\usepackage{amsmath, amsthm, amssymb, graphicx, here}

\newcounter{citedtheorems}

\newtheorem{defn}{Definition}[section]
\newtheorem{theorem}[defn]{Theorem}
\newtheorem*{th-ref}{Theorem} 
\newtheorem{thm-lit}[citedtheorems]{Theorem}  
\newtheorem{cor-lit}[citedtheorems]{Corollary}
\newtheorem{fact}[defn]{Fact}
\newtheorem{cor}[defn]{Corollary}

\newtheorem{concl}[defn]{Conclusion}
\newtheorem{conv}[defn]{Convention}
\newtheorem{claim}[defn]{Claim}
\newtheorem{lemma}[defn]{Lemma}
\newtheorem{obs}[defn]{Observation}
\newtheorem{rmk}[defn]{Remark}
\newtheorem{note}[defn]{Note}
\newtheorem{expl2}[defn]{Example}

\newtheorem{expl}[defn]{Example}

\newtheorem{qst}[defn]{Question}

\newcommand{\br}{\vspace{2mm}}

\newcommand{\vrt}{\rule{0pt}{12pt}}
\newcommand{\svrt}{\rule{0pt}{10pt}}

\newcommand{\loc}{\operatorname{Loc}}

\newcommand{\fss}{{\mathcal{P}}_{\aleph_0}}

\newcommand{\colct}{\operatorname{col-ct}}

\newcommand{\vp}{\varphi}

\begin{document}

\title{Persistence and NIP in the characteristic sequence}
\author{M. E. Malliaris}
\address{Group in Logic, University of California at Berkeley, 910 Evans Hall, Berkeley, CA 94720}
\email{mem@math.berkeley.edu}

\begin{abstract}
For a first-order formula $\varphi(x;y)$ we introduce and study 
the characteristic sequence $\langle P_n : n < \omega \rangle$ of hypergraphs defined by 
$P_n(y_1,\dots,y_n) := (\exists x) \bigwedge_{i \leq n} \varphi(x;y_i)$.   
We show that combinatorial and classification theoretic properties of the 
characteristic sequence reflect classification theoretic properties of $\varphi$ and vice versa.   
Specifically,  we show that  some  tree  properties are detected by the presence of certain 
combinatorial configurations in the characteristic sequence while other properties 
such as instability and the independence property manifest themselves in the persistence 
of complicated configurations under localization.
\end{abstract}

\maketitle

\section{Introduction}
This article defines and develops the theory of characteristic sequences. 
The characteristic sequence $\langle P_n : n<\omega \rangle$ associated to a first-order formula $\vp$ is a countable
sequence of hypergraphs defined on the parameter space of $\vp$; 
this association allows for a new description of the combinatorial complexity
of $\vp$-types in terms of graph-theoretic complexity of the hypergraphs $P_n$.
The construction arose from work of the author on saturation of ultrapowers, as described briefly below, 
but is of independent interest.
There is a model-theoretic sensibility throughout, but many of the arguments are combinatorial.
In fact, if the reader is familiar with basic model theory and is willing to take on faith the interest of certain classification-theoretic dividing lines, the article is largely self-contained. 

At first glance, the characteristic sequence gives a transparent language for the kinds of arguments which occur in many contexts where
the fine structure of dividing or thorn-dividing is being analyzed. However, there is a power in our general framework
which accrues from the fact that the formulas $P_n$ are simultaneously:

\begin{enumerate}
\item graphs, so we can ask about their complexity in the sense of graph theory;
\item formulas definable in the background theory $T$, so we can ask about their complexity in the sense of classification theory;
\item descriptors of the parameter space of the given formula $\vp$. 
\end{enumerate}

Leveraging these three contexts against each other puts strong restrictions on the behavior of the $P_n$. We obtain, for instance,
a description of NIP theories as theories in which any initial segment of the characteristic sequence is, 
after localization, essentially trivial: see \S 6 below. In some sense, then, the characteristic sequence 
is a tool for analyzing the fine structure of the independence property and gives a natural description of its complexity. 

By way of describing the kinds of complexity we consider, let us briefly mention
two background motivations for this work. The first is a deep question of Keisler 
about the structure of a preorder on countable theories which compares the difficulty of producing saturated
regular ultrapowers \cite{keisler}. Shelah described the structure of this so-called Keisler order on NIP theories in a series of 
surprising results, collected in \cite{Sh:c} chapter VI, but its structure on theories with the independence property
remains open. 
In \cite{mm-article}, \cite{mm-thesis} we showed that the structure of Keisler's order 
on unstable theories (thus, on theories with the independence property)
depends on a classification of $\vp$-types, and specifically on an analysis of characteristic sequences. 

The second motivation is the development of a new language for interactions between model theory and graph theory. 
An interest in the complexity of $\vp$-types asks how the many finite fragments of configurations 
cluster in the characteristic sequence and how uniformly or regularly they are distributed. 
These are issues which graph theory is particularly 
articulate at describing. Insofar as properties like edge density and edge distribution and structural properties of the 
hypergraphs $P_n$ can be shown to have model-theoretic content, this opens up the possibility of using a deep collection of
structure theorems for graphs  
to give model-theoretic information \cite{mm-thesis}. 

The organization of the article is as follows. Section 2 contains definitions and basic properties. Section 3 gives several
motivating examples. Section 4 gives a series of ``static'' arguments relating configurations in the characteristic sequence to
classification-theoretic dividing lines. Section 4 begins the work of separating out inessential complexity from
the base set of a type under analysis via localization; 
persistence and its associated ``dynamic'' arguments are motivated and defined. 
Section 6 describes NIP, simplicity and stability in terms of persistence. 

\subsection*{Acknowledgments} This is work from my doctoral thesis at Berkeley under the direction of Thomas Scanlon. I would also like to thank
Leo Harrington for many interesting discussions and John Baldwin for helpful remarks on an early version. Part of this work was done as a
Mathlogaps fellow in Lyon and I thank the \'Equipe de Logique for their warm hospitality.

\section{The characteristic sequence} 

\begin{defn} \label{notation} \emph{(Notation and conventions, I)}
\begin{enumerate}
\item Throughout this article, if a variable or a tuple is written $x$ or $a$ rather than $\overline{x}, \overline{a}$, this does not necessarily
imply that $\ell(x), \ell(a)=1$. 
\item Unless otherwise stated, $T$ is a complete theory in the language $\mathcal{L}$.
\item A graph in which no two elements are connected is called an \emph{empty graph}. A pair of elements which are not
connected is an \emph{empty pair}. When $R$ is an $n$-ary edge relation, to say that some $X$ is an \emph{$R$-empty graph} means that
$R$ does not hold on any $n$-tuple of elements of $X$. $X$ is an \emph{$R$-complete graph} if $R$ holds on every $n$-tuple from $X$. 
\item \textbf{Important:} $\vp_n(x;y_1,\dots y_n)$ denotes the formula $\bigwedge_{i\leq n} \vp(x;y_i)$.
\item In discussing graphs we will typically write concatenation for union, i.e. $Ac$ for $A \cup \{c \}$.  
\item \label{div} A formula $\psi(x;y)$ of $\mathcal{L}$ will be called \emph{dividable} if there exists an infinite set $C \subset P_1$ 
and $k<\omega$ such that $\{ \psi(x;c) : c \in C \}$ is $1$-consistent but $k$-inconsistent. (Thus, by compactness, some instance of $\psi$
divides.)  
When it is important to specify the arity $k$, write \emph{$k$-dividable}. 
\item A set is \emph{$k$-consistent} if every $k$-element subset is consistent, and it is \emph{$k$-inconsistent} if every $k$-element
subset is inconsistent. 
\end{enumerate}
\end{defn} 

To each formula $\vp$ we associate a countable sequence of hypergraphs, the \emph{characteristic sequence}, which describe incidence relations on the
parameter space of $\vp$. The idea is to give an analysis of $\vp$-types by describing the way that certain distinguished sets $A$
(the complete $P_\infty$-graphs, avatars of consistent partial $\vp$-types) sit inside the ambient hypergraphs $P_n$.

\begin{defn} \emph{(Characteristic sequences)} \label{characteristic-sequence} 
Let $T$ be a first-order theory and $\vp$ a formula of the language of $T$. 
\begin{itemize}
\item For $n<\omega$, $P_n(z_1,\dots z_n) :=  \exists x \bigwedge_{i\leq n} \vp(x;z_i) $.
\item The \emph{characteristic sequence} of $\vp$ in $T$ is $\langle P_n : n<\omega \rangle$.
\item Write $(T,\vp) \mapsto \langle P_n \rangle$ for this association. 
\item Convention: we assume that $T \vdash \forall y \exists z \forall x (\vp(x;z) \leftrightarrow \neg \vp(x;y))$. If this does not already
hold for some given $\vp$, replace $\vp$ with $\theta(x;y,z) = \vp(x;y) \land \neg\vp(x;z)$. 
\end{itemize}
\end{defn}

\begin{conv} \label{localization-depends-on-T}
\emph{Below, we will ask a series of questions about whether certain, possibly infinite, configurations appear as subgraphs of the $P_n$, 
or of the $P^f_n$ in some finite localization, Definition \ref{localization}. 
For our purposes, the existence of these configurations is a property of $T$. 
That is, we may, as a way of speaking, ask if some configuration $X$ appears, or is persistent,
inside of some $P_n$; however, we will always mean \emph{whether or not it is consistent with $T$} 
that there are witnesses to $X$ inside of $P_n$ interpreted in some sufficiently saturated model.
Certainly, one could ask the question of whether some \emph{given} model of $T$, expanded to model of the $P_n$, must include 
witnesses to $X$; we will not do so here.   
Thus, the formulas $P_n$ will often w.l.o.g. be identified with their interpretations in some monster model.} 
\end{conv}

\begin{defn} \emph{(Notation and conventions, II)}
\begin{enumerate}
\item[(8)] $P_\infty$ will be shorthand for the collection of predicates $P_n$ when the context 
(of a given condition, not necessarily definable, which holds of $P_n$ for all $n$) is clear, e.g. $A$ is a $P_\infty$-complete graph
meaning $A$ is a $P_n$-complete graph for all $n$.
\item[(9)] The complete $P_\infty$-graph $A$ will be called a \emph{positive base set} when the emphasis is on its identification with some
consistent partial $\vp$-type under analysis, as described in Observation \ref{bp}(5). 
\item[(10)] The sequence $\langle P_n \rangle$ \emph{has support $k$} if: $P_n(y_1,\dots y_n)$
iff $P_k$ holds on every $k$-element subset of $\{ y_1,\dots y_n \}$. See Remark \ref{support-fcp}.
\item[(11)] The element $a \in P_1$ is a \emph{one-point extension} of the $P_n$-complete graph $A$ just in case $Aa$ is also 
a $P_n$-complete graph. In most cases, $n$ will be $\infty$. 
\end{enumerate}
\end{defn}

\begin{obs} \emph{(Basic properties)} \label{bp}
Let $\langle P_n : n<\omega \rangle$ be the characteristic sequence of $(T,\vp)$. Then, regardless of the choice of $T$ and $\vp$,
we will have:
\begin{enumerate}
\item (Reflexivity) $\forall x ( P_1(x) \rightarrow P_n(x,\dots x) )$. 
In general, for each $\ell \leq m < \omega$, 
\begin{align*} \forall z_1,\dots z_\ell, y_1, \dots y_m  & \left( \vrt \left( \{ z_1,\dots z_\ell \} = \{ y_1, \dots y_m \} \right) \right. \\  
\implies & \left. \left( P_\ell(z_1,\dots z_\ell) \iff P_m (y_1,\dots y_m) \right) \vrt \right) \\
\end{align*}
\item (Symmetry) For any $n < \omega$ and any bijection $g: n \rightarrow n$,
\[ \forall y_1,\dots y_n \left(\vrt P_n(y_1,\dots y_n) \iff P_n(y_{g(1)}, \dots  y_{g(n)})\right) \]
\item (Monotonicity) 
For each $\ell \leq m <\omega$,
\begin{align*}
\forall z_1,\dots z_\ell, y_1, \dots y_m & \left( \vrt \left(  \{ z_1, \dots z_\ell \} \subseteq \{ y_1,\dots y_m \} \right) \right. \\
 \implies & \left. \left( P_m (y_1,\dots y_m) \implies P_\ell(z_1,\dots z_\ell) \right) \vrt \right) \\
\end{align*}
\noindent So in particular, if $\models P_m(y_1,\dots y_m)$ and $\ell < m$ then $P_\ell$ holds on all $\ell$-element subsets of $\{y_1,\dots y_m \}$.
The converse is usually not true; see Remark \ref{support-fcp}.

\item (Dividing) Suppose that for some $n <\omega$, it is consistent with $T$ that 
there exists an infinite subset $Y \subset P_n$
such that $Y^k \cap P_{nk} = \emptyset$. Then in any sufficiently saturated model of $T$, 
some instance of the formula $\vp_n(x;y_1,\dots y_n) = \bigwedge_{i<n} \vp(x;y_i)$
$k$-divides. 
\item (Consistent types) Let $A \subset P_1$ be a set of parameters in some $M \models T$.  Then 
\\ $\{ \vp(x;a) : a \in A \}$ is a consistent partial $\vp$-type iff $A^n \subset P_n$ for all $n<\omega$. 
\end{enumerate}
\end{obs}

\begin{proof} 
(4) By compactness, there exists an infinite indiscernible sequence of $n$-tuples
$C = \langle c^i_1,\dots c^i_n : i<\omega \rangle$ such that $C^k \cap P_{nk} = \emptyset$. 
The set $\{ \vp_n(x;c^i_1,\dots c^i_n) : i<\omega  \}$ is therefore $k$-inconsistent. 
However, it is $1$-consistent: for each $c^i_1,\dots c^i_n \in C$,
$M \models P_n(c^i_1,\dots c^i_n)$, so $M \models \exists x \vp_n(x;c^i_1,\dots c^i_n)$.
\end{proof}

\begin{conv} \label{conv-configurations} \emph{($T_0$-configurations)}
Throughout this article, let $T_0$ denote the incomplete theory in the language $\mathcal{L}_0 := \{ P_n : n<\omega \} \cup \{ = \}$
which describes (1)-(3) of Observation \ref{bp}. Blueprints for hypergraphs in the language $\mathcal{L}_0$ which are
consistent with $T_0$ will be called {$T_0$-configurations}. 
That is: a \emph{finite $T_0$-configuration} is a pair $X = (V_X, E_X)$ where 
$V_X = n < \omega$, $E_X \subseteq \mathcal{P}(n)$ and the following is consistent with $T_0$:
\begin{equation} \label{tzero-eqn}
(\exists x_1,\dots x_n) ~ (\forall \sigma \subseteq n, |\sigma| = i, \sigma = \{{\ell_1},\dots {\ell_i} \})~
\left( \vrt P_i(x_{\ell_1},\dots x_{\ell_i}) \iff \sigma \in E_X  \right)
\end{equation}

\noindent In general, the domain of a $T_0$-configuration may be infinite; we simply require that its restriction
to every finite subdomain satisfy $(\ref{tzero-eqn})$.  
These are the graphs which can consistenly occur as finite subgraphs of \emph{some} characteristic sequence. 
That every such graph appears in some sequence follows from Example \ref{maximal-example} below.
\end{conv}

\begin{conv} \label{conv-configurations-1} \emph{($T_1$-configurations)}
Fix $T, \vp$, and the associated sequence $\langle P_n : n<\omega \rangle$. 
Let $M \models T$; there is a unique expansion of $M$ to $\mathcal{L}_0 = \{ P_n : n<\omega \} \cup \{ = \}$. 
Throughout this article, whenever $T, \vp, \langle P_n \rangle$ are thus fixed,
let $T_1$ denote the complete theory of $M$ in the language $\mathcal{L}_0$. As the characteristic sequence
is definable in $T$, when $T$ is complete this will not depend on the model chosen. 

Hypergraphs in the language $\mathcal{L}_0$ which are
consistent with $T_1$ will be called \emph{$T_1$-configurations}. 
\end{conv}

Recall that a formula $\vp(x;y)$ has the \emph{finite cover property} if for arbitrarily large $n<\omega$
there exist $a_0,\dots a_n$ such that $\{ \vp(x;a_0),\dots \vp(x;a_n) \}$ is $n$-consistent but $(n+1)$-inconsistent. 
 
\begin{rmk} \label{support-fcp}
The following are equivalent, for $(T,\vp) \mapsto \langle P_n \rangle$:
\begin{enumerate}
\item There is $k<\omega$ such that the sequence $\langle P_n \rangle$ has support $k$. 
\item $\vp$ does not have the finite cover property. 
\end{enumerate}
\end{rmk}

\noindent 
In practice, when analyzing saturation of $\vp$-types the finite cover property can often, but not always, be avoided by a judicious choice of formula.
For instance, if $\vp$ is unstable, some fixed finite conjunction $\theta$ of instances of $\vp$ has the finite cover property
(\cite{Sh:c}.II.4); if we choose to present $\vp$-types as $\theta$-types the
characteristic sequence would not have finite support.
Nonetheless, it may happen even in unstable theories that there is a set $\Sigma \subset \mathcal{L}$ of formulas without the fcp such that
$M \models T$ is $\lambda^+$-saturated iff $M$ realizes all $\vp_0$-types over sets of size $\lambda$ for all $\vp_0 \in \Sigma$. 
This is true, for instance, of $\Sigma = \{ \psi(x;y,z) := xRy \land \neg xRz \}$ in the random graph, and of 
$\Sigma = \{ \psi(x;y,z) := y < x < z \}$ in $(\mathbb{Q}, <)$.

\section{Some examples}

This section works out several motivating examples. 
We refer informally to localization and persistence, which will be defined in Definitions \pageref{localization} and \pageref{persistence} below; 
the general definitions of $(\eta, \nu)$-arrays and trees will be given in Definition \ref{diagrams-arrays}.

\begin{expl2}(The random graph) \label{expl-r-g} 
\end{expl2}
$T$ is the theory of the random graph, and $R$ its binary edge relation. 
Let $\vp(x;y,z) = xRy \land \neg xRz$, with $(T,\vp) \mapsto \langle P_n \rangle$. Then:

\begin{itemize}
\item $P_1((y,z)) \iff y\neq z$.
\item $P_n((y_1,z_1),\dots (y_n,z_n)) \iff \{ y_1,\dots y_n \} \cap \{z_1,\dots z_n\} = \emptyset$.
\end{itemize}

Notice:

\begin{enumerate}
\item The sequence has support 2.
\item There is a uniform finite bound on the size of an empty graph $C \subset P_1, C^2 \cap P_2 = \emptyset$: an analysis of the theory
shows that $\vp$ is not dividable, and inspection reveals this bound to be 3.
\item $P_n$ does not have the order property for any $n$ and any partition of the $y_1,\dots y_n$ into object and
parameter variables. (Proof: The order property in $P_n$ implies dividability of $\vp_{2n}$ by Observation \ref{order-dividable}. But
none of the $\vp_\ell$ are dividable, as inconsistency only comes from equality.)
\item Of course, the formula $\vp$ has the independence property in $T$. We can indeed find a configuration in $P_2$ which
witnesses this: any $C$ which models the $T_0$-configuration having $V_X=\omega$ and $\{i,j\} \notin E_X \iff \exists n(i=2n \land j=2n+1)$. 
Note that $\vp$ will have the independence property on any infinite
$P_2$-complete subgraph of the so-called $(\omega,2)$-array $C$ (see Observation \ref{omega-k-ip} below).  
\item As $\vp$ is unstable, $\vp$-types are not necessarily definable in the sense of stability theory. 
However, we can obtain a kind
of definability ``modulo'' the independence property, or more precisely, 
definability over the name for a maximal consistent subset of an $(\omega, 2)$-array as follows: 
\end{enumerate}

\br
\noindent \emph{Definable types modulo independence.}
Let $p \in S(M)$ be a consistent partial $\vp$-type presented as a positive base set $A \subset P_1$. 
Let us suppose $p \vdash \{xRc : c\in C\} \cup \{ \neg xRd : d \in D \} \vdash p$,
so that $A \subset M^2$ is a collection of pairs of the form $(c, d)$ which generate the type. 

There is no definable (in $T$ with or without parameters, so in particular not from $P_2$) extension of the type $A$, so we cannot expect
to find a localization of $P_1$ around $A$ which is a $P_2$-complete graph. However:

\begin{claim} In the theory of the random graph, with $\vp(x;y,z) = xRy \land \neg xRz$ as above,
for any positive base set $A \subset P_1$
there exist a definable $(\omega,2)$-array $W \subset P_1$, a solution $S$ of $W$ and an $S$-definable $P_\infty$-graph
containing $A$. 
\end{claim}

\begin{proof}
Work in $P_1$. Fix any element $(a,b)$ with $a,b \notin C,D$ and set $W_0 := \{ (y,z) \in P_1 : \neg P_2((y,z), (a,b)) \}$. 
Thus $W_0 = \{ (b,z) : z \neq b \} \cup \{ (y,a) : y \neq a \}$. So the only $P_2$-inconsistency among elements of $W_0$ comes from 
pairs of the form $(b,c), (c,a)$; thus, writing Greek letters for the elements of $P_1$,
\[ (\forall \eta \in W_0) (\exists \nu \in W_0)(\forall \zeta \in W_0)\left(\neg P_2(\eta, \zeta) \rightarrow \zeta = \nu \right) \] 
\vspace{1mm}

\noindent In other words, $W := W_0 \setminus \{(b,a)\}$ is an $(\omega, 2)$-array (Definition \ref{diagrams-arrays}). Moreover:
\br
\begin{enumerate}
\item $(y,z),(w,v) \in W$ and $\neg P_2((y,z), (w,v))$ implies $y=v$ or $z=w$, and
\item for any $c \neq a,b$, there are $d,e \in M$ such that $(d,c), (c,e) \in W$. Thus:
\item we may choose a maximal complete $P_2$-subgraph $C$ of $W$ such that $CA$ is a complete $P_\infty$-graph. 
For instance, let $C$ be any maximal complete extension of $\{ (b,d) : d \in D \} \cup \{ (c, a) : c \in C \}$. 
Call any such $C$ a \emph{solution} of the array $W$.  
\end{enumerate}
\br
Let $S$ be a new predicate which names this solution $C$ of $W$. Then 
$\{ y \in P_1 : z \in S \rightarrow P_2(y,z) \} \supset A$ is a $P_2$-complete graph, definable in $\mathcal{L} \cup \{S\}$.
Support 2 implies that it is a $P_\infty$-graph. 
Notice that by (2), we have in fact chosen a maximal consistent extension of $A$ (i.e. a complete global type).
\end{proof} 

\begin{rmk}
The idiosyncracies of this proof, e.g. the choice of a \emph{definable} $(\omega,2)$-array, reflect an interest in structure
which will be preserved in ultrapowers. 
\end{rmk}

\begin{expl2} (Coding complexity into the sequence) \label{coding-complexity} 
\end{expl2} 
It is often possible to choose a formula $\vp$ so that some particular configuration appears in its characteristic sequence. 
For instance, by applying the template below when $\vp$ has the independence property, 
we may choose a simple unstable $\theta$ whose $P_2$ is universal for finite bipartite graphs
$(X,Y)$, provided we do not specify whether or not edges hold between $x,x^\prime \in X$ or between $y, y^\prime \in Y$. 
Nonetheless, Conclusion \ref{ps-stable} below will show this is ``inessential'' structure in the case of simple theories: 
whatever complexity was added through coding can be removed through localization.  

\emph{The construction.}
Fix a formula $\vp$ of $T$. Let $\theta(x;y,z,w) :=  (z=w \land x=y) \lor (z\neq w \land \vp(x;y))$. 
Write $(y,*)$ for $(y,z,w)$ when $z=w$, and $(y,-)$ for $(y,z,w)$ when $z\neq w$. 
Let $\langle P_n \rangle$ be the characteristic sequence of $\theta$,
$\langle P^\vp_n \rangle$ be the characteristic sequence of $\vp$,
and $\langle P^=_n \rangle$ be the characteristic sequence of $x=y$.
Then $P_n$ can be described as follows:

\begin{itemize}
\item  $P_n((y_1,-), \dots (y_n,-)) \leftrightarrow P^\vp_n(y_1,\dots y_n)$. 

\item $P_n((y_1,*), \dots (y_n,*)) \leftrightarrow P^=_n(y_1,\dots y_n)$. 

\item Otherwise, the $n$-tuple $ y := ((y_1,z_1),\dots (y_n, z_n))$ can contain
(up to repetition) at most one $*$-pair, so $z_i = z_j = *$ $\rightarrow y_i = y_j$. 
In this case the unique $y_*$ in the $*$-pair is the realization of some $\vp$-type
in the original model $M$ of $T$, and 
\\ $P_{n+1}((y_*, *), (y_1,-),\dots (y_n, -))$ 
holds iff $M \models \bigwedge_{j\leq n} \vp(y^*; y_j)$. 
\end{itemize}

\begin{rmk} \label{rmk-coding-complexity}
This highlights an important distinction: the fact that a characteristic sequence may contain a bipartite graph is not anywhere
near as powerful as the fact of containing a random graph, see Example \ref{maximal-example} below. 
In the coding just given we could not choose how elements within each side of the graph interrelated. This is quite restrictive, and 
eludes our coding for deep reasons: for instance, applying a consistency result of Shelah on the Keisler order one can show that 
the order property in the characteristic sequence cannot imply the \emph{compatible} order property in the characteristic sequence, Definition \ref{c-op} below \cite{mm-thesis}. 
\end{rmk}

\begin{expl2} 
(A theory with $TP_2$) \label{ex-tree} 
\end{expl2} 
$TP_2$ is Shelah's tree property of the second kind, to be defined and discussed in detail in Definition \ref{TP1-TP2}.
Let $T$ be the model completion of the following theory \cite{ShUs}. There are two infinite sorts $X, Y$ and a single
parametrized equivalence relation $E_x(y,z)$, where $x \in X$, and $y,z \in Y$. 
Let $\vp_{eq} := \vp(y;xzw) = E_x(y,z) \land \neg E_x(z,w)$. Then: 

\begin{itemize}
\item $P_1 ((xzw)) \iff z\neq w$.
\item $P_2 ((x_1 z_1 w_1), (x_2 z_2 w_2)) \iff$ each triple is in $P_1$ and furthermore:
\[ (x_1 = x_2) \rightarrow (  ~E_x(z_1,z_2) \land \bigwedge_{i\neq j \leq 2} \neg E_x(w_i, z_j) ) \]
\end{itemize}

The sequence has support 2. 
There are many empty graphs; these persist under localization (Theorem \ref{char-simple}). 
One way to see the trace of $TP_2$ is as follows. 
Fixing $\alpha$, choose $a_i ~ (i<\omega)$ to be a set of representatives
of equivalence classes in $E_\alpha$, and choose $b$ such that $\neg E_\alpha(a_i,b) ~(i<\omega)$. Then 
$\{ (\alpha, a_i, b) : i<\omega \} \subset P_1$ is a $P_2$-empty graph. We in fact have arrays $\{ (\alpha^t, a^t_i, b^t): i<\omega, t<\omega \}$
whose ``columns'' (fixing $t$) are $P_2$-empty graphs and where every path which chooses exactly one element from each column is 
a $P_2$-complete graph, thus a $P_\infty$-complete graph. The parameters in this so-called $(\omega, \omega)$-array describe $TP_2$ for $\vp_{eq}$
(Claim \ref{omega-omega-trees}).   

Note that a gap has appeared between the classification-theoretic complexity of $\vp$, which is not simple, and that of the formula $P_2$:

\begin{claim} $P_2$ does not have the order property.
\end{claim}

\begin{proof} 
This is essentially because inconsistency requires the parameters $x$ to coincide.
Suppose that $\langle a_i, b_i : i<\omega \rangle$ were a witness to the order property for $P_2$. Fix any $a_i =(\alpha_s, a_s, d_s)$. 
Now $\neg P_2(b_j, a_i)$ for $j<i$, where $b_j = (\beta_t, b_t, c_t)$. $P_2$-inconsistency requires $\alpha_s = \beta_t$. As this is 
uniformly true, $\alpha_s=\alpha_t=\beta_s=\beta_t$ for all $s,t <\omega$ in the sequence. 
But now that we are in a single equivalence relation $E_\alpha$,
transitivity effectively blocks order: $\neg P_2(b_j, a_i) \leftrightarrow \neg E_\alpha(a_s, b_t)$. Depending on whether at least one of the
$a$- or $b$-sequences is an empty graph, we can find a contradiction to the order property with either three or four elements. 
\end{proof}

\begin{expl2} (A maximally complicated theory) \label{maximal-example}
\end{expl2}
In this example the sequence is universal for finite $T_0$-configurations (Convention \ref{conv-configurations}),   
a natural sufficient condition for ``maximal complexity.''

Let the elements of $M$ be all finite subsets of $\omega$; the language has two binary relations, $\subseteq$ and $=$,
with the natural interpretation. Set $T = Th(M)$. 

Choose $\vp_\subseteq := \vp(x;y,z) = x \subseteq y \land x \not\subseteq z$. Then:

\begin{itemize}
\item $P_1((y,z)) \iff \emptyset \subsetneq y \not\subseteq z$.
\item $P_n((y_1,z_1),\dots (y_n,z_n)) \iff \emptyset \subsetneq \bigcap_{i\leq n} y_i \not\subseteq \bigcup_{i\leq n} z_i$.
\end{itemize}

The sequence does not have finite support. Moreover:

\begin{claim}
Let $\langle P_n \rangle$ be the characteristic sequence of $\vp_\subseteq$, $k<\omega$, and
let $X$ be a finite $T_0$-configuration. Then there exists a finite $A \subseteq P_1$
witnessing $X$.  
\end{claim}

\begin{proof}
Write the elements of $P_1$ as $w_i = (y_i, z_i)$;
it suffices to choose the positive pieces $y_i$ first, and afterwards take the $z_i$ to be completely disjoint.
More precisely, suppose $X$ is given by $V_X = m$ and $E_X \subset \mathcal{P}(m)$. We need simply to choose
$y_1,\dots y_m$ such that for all $\sigma \subseteq m$,
\[   \left( \bigcap_{j\in\sigma} y_j \neq \emptyset \right)  \iff \sigma \in E_X \]
which again, is possible by the downward closure of $E_X$.
\end{proof}

\begin{cor}
This characteristic sequence is universal for finite $T_0$-configurations.  
\end{cor}

\begin{rmk} \label{random-suffices}
That the sequence is universal for finite $T_0$-configurations is sufficient, though not necessary, for maximal complexity 
in the Keisler order. By \cite{Sh:c}.VI.3, $\vp(x;y,z) = y < x <z$ in $Th(\mathbb{Q}, <)$ is maximal. 
Its characteristic sequence has support 2, but its $P_2$ is clearly not universal. 
\end{rmk}

\section{Static configurations} \label{section:static-config}

This section establishes a series of correspondences between $T_1$-configurations found in the characteristic sequence
of $\vp$ and the classification-theoretic complexity of $\vp$ itself. It lays the groundwork for the next section,
which will build ``dynamic'' arguments out of these ``static'' ones by asking what happens when certain configurations
persist under all reasonable restrictions of the set $P_1$. 

Here we describe configurations which signal that $\vp$ has
the order property, the independence property, the tree property and $SOP_2$. Recall that:

\begin{defn} \emph{(Tree properties)} \label{TP1-TP2} Let $\subseteq$ indicate initial segment. 
To simplify notation, say that the nodes $\rho_1, \rho_2 \in \omega^{<\omega}$ are \emph{$^*$incomparable} if
\[ \neg(\rho_1 \subseteq \rho_2) \land \neg(\rho_2 \subseteq \rho_1) \land \neg (\exists \nu \in \omega^{<\omega}, i,j \in \omega)(\rho_1 = \nu^\smallfrown i, \rho_2=\nu^\smallfrown j)  \]
i.e., if they do not lie along the same branch and are not immediate successors of the same node. 

\br
\noindent Then the formula $\vp$ has:
\begin{itemize}

\item the \emph{$k$-tree property}, where $k<\omega$, if there is an $\omega^{<\omega}$-tree of instances of $\vp$ where paths are consistent
and the immediate successors of any given node are $k$-inconsistent, i.e. $X = \{ \vp(x; a_\eta) : \eta \in \omega^{<\omega} \}$, and:

\begin{enumerate}
\item for all $\nu \in \omega^\omega$, $\{ \vp(x;a_\eta) : \eta \subseteq \nu \}$ is a consistent partial type;
\item for all $\rho \in \omega^{<\omega}$, $\{ \vp(x;a_{\rho^\smallfrown i}): i<\omega \}$ is $k$-inconsistent.
\end{enumerate}

Call any such $X$ a \emph{$\vp$-tree}, or if necessary a $\vp$-$k$-tree.

\item the \emph{tree property} if it has the $k$-tree property for some $2 \leq k < \omega$. 

\item the \emph{non-strict tree property} \emph{$TP_2$} if there exists a $\vp$-tree with $k=2$ and for which, moreover: 

\begin{enumerate}
\item[$(3)_2$] for any two $^*$incomparable $\rho_1, \rho_2 \in \omega^{<\omega}$,
$\exists x( \vp(x; a_{\rho_1}) \land \vp(x; a_{\rho_2})  )$. 
\end{enumerate}

\item the \emph{strict tree property}, also known as $TP_1$ or $SOP_2$, if there exists a $\vp$-tree with $k=2$ and for which, moreover: 

\begin{enumerate}
\item[$(3)_1$] for any two $^*$incomparable $\rho_1, \rho_2 \in \omega^{<\omega}$, $\neg \exists x(\vp(x;a_{\rho_1}) \land \vp(x;a_{\rho_2}))$. 
\end{enumerate}

\end{itemize}
\end{defn}

\begin{thm-lit} \emph{(Shelah; see \cite{Sh:c}.III.7}) \label{thm-trees}
\begin{itemize}
\item $T$ is simple iff no formula $\vp$ of $T$ has the tree property, iff no $\vp$ has the $2$-tree property. 
\item If $\vp$ has the $2$-tree property then either $\vp$ has $TP_1$ or $\vp$ has $TP_2$. 
\end{itemize}
\end{thm-lit}

\noindent We fix a monster model $M$ from which the parameters are drawn; see Convention \ref{localization-depends-on-T}. 

\begin{defn} \emph{(Diagrams, arrays, trees)} \label{diagrams-arrays}
Let $\lambda \geq \mu$ be finite cardinals or $\omega$. Write $\subseteq$ to indicate initial segment.  
The sequence $\langle P_n \rangle$ has:

\begin{enumerate}
\item  \emph{an $(\omega,2)$-diagram} if there exist elements $\{ a_\eta : \eta \in 2^{<\omega}\} \subseteq P_1$
such that 
\begin{itemize}
\item for all $\eta \in 2^{<\omega}$, $\neg P_2(a_{\eta^\smallfrown 0}, a_{\eta^\smallfrown 1})$, and
\item for all $n < \omega$ and $\eta_1,\dots \eta_n \in 2^{<\omega}$,  we have that $\eta_1 \subseteq \dots \subseteq \eta_n \implies P_n(a_{\eta_1},\dots a_{\eta_n})$
\end{itemize}

\noindent That is, sets of pairwise comparable elements are $P_\infty$-consistent, while immediate successors of the same node are $P_2$-inconsistent.

\br
\item 
\emph{a $(\lambda, \mu, 1)$-array} if 
there exists $X = \{ a^m_{l} : l < \lambda, m < \mu \} \subset P_1$
such that: 

\begin{itemize}
\item $P_2(a^{m_1}_{l_1}, a^{m_2}_{l_2}) \iff \left(l_1 = l_2 \rightarrow m_1 = m_2 \right)$
\item For all $i < \omega$, 
\[ P_n(a^{m_1}_{l_1}, \dots a^{m_n}_{l_n}) \iff \bigwedge_{1\leq i,j \leq n} P_2(a^{m_i}_{l_i}, a^{m_j}_{l_j}) \]
\end{itemize}

\noindent That is, any $C \subset X$, possibly infinite, is a $P_\infty$-graph iff it contains no more than one element
from each column. (We will relax this last condition in the more general Definition \ref{omega-n-arrays} below.) 

\br
\item 
\emph{a $(\lambda, \mu)$-tree} if there exist elements $\{ a_\eta : \eta \in \mu^{<\lambda} \} \subset P_1$
such that 
\begin{itemize}
\item for all $\eta_2, \eta_2 \in \mu^{<\lambda}$, 
\[ P_2(a_\eta, a_\nu) \iff \left(\eta_1 \subseteq \eta_2 \lor \eta_2 \subseteq \eta_1 \right) \]
\noindent i.e. \emph{only if} the nodes are comparable; and
\item for all $n < \omega, \eta_1,\dots \eta_n \in \mu^{<\lambda}$, 
\[ \eta_1 \subseteq \dots \subseteq \eta_n  \implies P_n(a_{\eta_1},\dots a_{\eta_n}) \]
\end{itemize}
\end{enumerate}
\end{defn}

\begin{rmk}
Diagrams are prototypes which can give rise to either arrays or trees, in the case where the unstable formula $\vp$ has
the independence property or $SOP_2$, respectively. 

The arrays will be revisited in Definitions \ref{omega-2-arrays} and \ref{omega-n-arrays}.
\end{rmk}

\begin{claim} \label{diagrams-instability}
Let $\vp$ be a formula of $T$ and set $\theta(x;y,z) = \vp(x;y) \land \neg \vp(x;z)$.
Let $\langle P_n \rangle$ be the characteristic sequence of $(T,\theta)$. 
The following are equivalent:

\begin{enumerate}
\item $\langle P_n \rangle$ has an $(\omega, 2)$-diagram.
\item $R(x=x, \vp(x;y), 2) \geq \omega$, i.e. $\vp$ is unstable.
\item $R(x=x, \theta(x;yz), 2) \geq \omega$, i.e. $\theta$ is unstable.
\end{enumerate}
\end{claim}

\begin{proof}
(2) $\rightarrow$ (1): We have in hand a tree of partial $\vp$-types $\mathcal{R} = \{ p_\nu : \nu \in 2^\omega \}$, 
partially ordered by inclusion, witnessing that $R(x=x,\vp,2) \geq \omega$. Let us show that we can build an $(\omega, 2)$-diagram.
That is, we shall choose parameters $\{ a_\eta : \eta \in 2^{<\omega} \} \subset P_1$ satisfying Definition \ref{diagrams-arrays}(1).

First, by the definition of the rank $R$, which requires the partial types to be explicitly contradictory, 
we can associate to each $\nu$ an element $c_\nu \in M$, $\ell(c_\nu) = \ell(y)$ such that:
\begin{itemize}
\item $\vp(x;c_\nu) \in p_{\nu^\smallfrown 1} \setminus p_{\nu}$, and
\item $\neg \vp(x;c_\nu) \in p_{\eta^\smallfrown 0} \setminus p_{\eta}$.
\end{itemize}
\noindent i.e., the split after index $\nu$ is explained by $\vp(x;c_\nu)$.

Second, choose a set of indices $\mathcal{S} \subseteq 2^{<\omega}$ such that:
\begin{itemize}
\item $(\forall \eta \in 2^{<\omega}) ~(\exists s \in \mathcal{S}) (\eta \subsetneq s) $
\item $(\forall s_1 \subsetneq s_2 \in \mathcal{S}) ~(\exists\eta \notin \mathcal{S}) ~(s_1 \subsetneq \eta \subsetneq s_2) $
\end{itemize}

It will suffice to define $a_{s^\smallfrown i}$ for $s \in \mathcal{S}$, $i \in \{0,1\}$. 
(The sparseness of $S$ ensures the chosen parameters for $\vp$ won't overlap, which will make renumbering straightforward.)
Recall that the $a_\eta$ will be parameters for $\theta(x;y,z) = \vp(x;y) \land \neg(x;z)$. So we define:
\begin{itemize}
\item $a_{s^\smallfrown 0} = (c_{s^\smallfrown 0}, c_s)$;
\item $a_{s^\smallfrown 1} = (c_s, c_{s^\smallfrown 1})$. 
\end{itemize}

The consistency of the paths through our $(\omega, 2)$-diagram is inherited from the tree $\mathcal{R}$ of consistent partial types. 
However, $\neg P_2(a_{s^\smallfrown 0}, a_{s^\smallfrown 1})$ because these contain an explicit contradiction:
\[ \neg \exists x \left( \vrt (\vp(x;c_{s^\smallfrown 0}) \land \neg \vp(x;c_s)) \land (\vp(x;c_s) \land \neg(\vp(x;c_{s^\smallfrown 1}) \right)  \]

(1) $\rightarrow$ (3): Reading off the parameters from the diagram we obtain a tree of consistent partial $\theta$-types
$\{ p_\eta : \eta \in 2^{<\omega} \}$, partially ordered by inclusion. For any $\eta \in 2^{<\omega}$, 
$\neg P_2(a_{\eta^\smallfrown 0}, a_{\eta^\smallfrown 1})$, i.e. 
$\neg \exists x(\theta(x;a_{\eta^\smallfrown 0}) \land \theta(x;a_{\eta^\smallfrown 1}))$. 
Furthermore, $\theta(x;a_{\eta^\smallfrown 0}) \in p_{\eta^\smallfrown 0} \setminus p_\eta$,
while $\theta(x;a_{\eta^\smallfrown 1}) \in p_{\eta^\smallfrown 1}\setminus p_\eta$. So there is no harm in making the types
explicitly inconsistent, as the rank $R$ requires, by adding $\neg~ \theta(x;a_{\eta^\smallfrown i})$ to $p_{\eta^\smallfrown j}$ for
$i\neq j < 2$.

(2) $\leftrightarrow$ (3): for all $A$, $|A| \geq 2$, $|S_\vp(A)| = |S_\theta(A)|$. 
\end{proof}

\begin{claim} \label{ip-in-cs}
Let $\vp$ be a formula of $T$ and set $\theta(x;y,z) = \vp(x;y) \land \neg \vp(x;z)$.
Let $\langle P_n \rangle$ be the characteristic sequence of $(T,\theta)$. 
The following are equivalent:
\begin{enumerate}
\item $\langle P_n \rangle$ has an $(\omega, 2, 1)$-array.
\item $\vp$ has the independence property. 
\item $\theta$ has the independence property.  
\end{enumerate}
\end{claim}

\begin{proof}
(1) $\rightarrow$ (3): This is Observation \ref{omega-k-ip}. (Essentially, let $A_0$ be the top row of the array $A$,
and $\sigma, \tau \subset A$ finite disjoint; let $B \subset A$ be a maximal positive base set, i.e. a maximal $P_\infty$-complete graph, in $A$ 
containing $\sigma$ and avoiding $\tau$. Then any realization of the type corresponding to $B$ is a witness to this instance of independence.)

(2) $\rightarrow$ (1): Let $\langle i_\ell : \ell<\omega \rangle$ be a sequence over which $\vp$ has the
independence property. For $t<2, j<\omega$ set $a^0_j = (i_\ell, i_{\ell+1})$, $a^1_j = (i_{\ell+1}, i_\ell)$. 
Then $\{ a^t_j : t<2, j<\omega \}$ is an $(\omega, 2, 1)$-array for $P_\infty$.  

(3) $\rightarrow$ (2): For any infinite $A$, $|S_\vp(A)| = |S_\theta(A)|$, as any type on one side can be presented as a type
on the other.  The independence property can be characterized in terms of the cardinality of the space of types over finite sets 
(\cite{Sh:c} Theorem II.4.11).
\end{proof}

\begin{claim} \label{omega-omega-trees}
Let $\vp$ be a formula of $T$ and set $\theta(x;y,z) = \vp(x;y) \land \neg \vp(x;z)$.
Let $\langle P_n \rangle$ be the characteristic sequence of $(T,\theta)$. 
Suppose that $T$ does not have $SOP_2$. Then the following are equivalent:
\begin{enumerate}
\item $\langle P_n \rangle$ has an $(\omega, \omega, 1)$-array.
\item $\vp$ has the $2$-tree property. 
\end{enumerate}
\end{claim}

\begin{proof}
(1) $\rightarrow$ (2) Each column (=empty graph) of the array witnesses that $\varphi$ is 2-dividable, 
and the condition that any subset of the array containing no more than one element from each column is a $P_\infty$-complete graph 
ensures that the dividing can happen sequentially. 

(2) $\rightarrow$ (1)
By Theorem \ref{thm-trees} above, $NSOP_2$ implies $\vp$ has $TP_2$. That is, there is a tree of instances
$\{ \vp(x;a_\eta) : \eta \in \omega^{<\omega} \}$ such that first, for any finite $n$, $\eta_1 \subseteq \dots \subseteq \eta_n$ implies that
the partial type $\{ \vp(x;a_{\eta_1}), \dots \vp(x;a_{\eta_n}) \}$ is consistent; and second,
\[  \neg \exists x \left( \vrt \vp(x;a_\eta) \land \vp(x;a_\nu) \right)  ~~\iff~~ (\exists\rho \in \omega^{<\omega})(\exists i \neq j \in \omega) 
\left(\vrt \eta = \rho^\smallfrown i  ~\land~ \nu = \rho^\smallfrown j  \right) \]
Thus the parameters $\{ a_\eta : \eta \in \omega^{<\omega} \} \subset P_1$ form an $(\omega, \omega,1 )$-array for $P_\infty$. 
\end{proof}

It is straightforward to characterize the analogous $k$-tree properties in terms of arrays whose columns are $k$-consistent but
$(k+1)$-inconsistent. 

\begin{claim} The following are equivalent:
\begin{enumerate}
\item $\langle P_n \rangle$ has an $(\omega, 2)$-tree.
\item $\vp$ has $SOP_2$. 
\end{enumerate}
\end{claim}

\begin{proof}
(2) $\rightarrow$ (1) This is a direct translation of Definition \ref{TP1-TP2}. 

(1) $\rightarrow$ (2) It suffices to show that $\langle P_n \rangle$ has an $(\omega, \omega)$-tree, 
which is true by compactness, using the strictness of the tree. 
\end{proof}

In the next definition, the power of the classification-theoretic order property on the $P_n$ 
is magnified when it can be taken to describe the interaction between complete graphs, i.e. base sets for partial $\vp$-types.
Compare Remark \ref{rmk-coding-complexity}.

\begin{defn} \label{c-op}
$P_\infty$ has the \emph{compatible order property} if there exists a sequence 
$C =\langle a_i, b_i : i<\omega \rangle \subset P_1$ such that for any $n<\omega$ and any $a_1, b_1,\dots a_n, b_n \subset C$, 
\[ P_n((a_1, b_1),\dots (a_n, b_n)) \iff  \left( \operatorname{max} \{ a_1, \dots a_n \} < \operatorname{min} \{ b_1, \dots b_n \} \right)\] 
\noindent Say that \emph{$P_m$} has the compatible order property to indicate that this holds for all $P_n$, $n\leq m$.
\end{defn}

\begin{obs} \label{o-tree}
Suppose $(T,\vp) \mapsto \langle P_n \rangle$, and that $\langle P_n \rangle$ has the
compatible order property. Then $\vp_2$ has the tree property, and in particular, $SOP_2$. 
\end{obs}

\begin{proof}
Let us build an $SOP_2$-tree $\{ \vp_2(x;a_\eta, b_\eta) : \eta \in \omega^{<\omega} \}$ following Definition \ref{TP1-TP2} above by specifying
the corresponding tree of parameters $\{ c_\eta : \eta \in \omega^{<\omega} \} \subset P_1$, where each $c_\eta$ is a pair $(a_\eta, b_\eta)$. 
Let $S = \langle a_i b_i : i< \mathbb{Q} \rangle$ be an indiscernible sequence witnessing the compatible order property. 
We will use two facts in our construction:

\begin{enumerate}
\item Let $\langle a_{i_\ell} b_{j_\ell} : \ell <\omega \rangle$ be any subsequence of $S$ such that 
$\ell < k \implies a_{i_\ell} < b_{j_\ell} < a_{i_k} < b_{j_k}$. 
Then $\{ \vp_2(x;a_{i_j},b_{i_j}) : j <\omega \}$ 2-divides by Observation \ref{op}. 
\item Let $a_{i_1},b_{j_1},\dots a_{i_n},b_{j_n} \in S$. Then
\[ P_n((a_{i_1},b_{j_1}),\dots (a_{i_n},b_{j_n})) \iff ~ \operatorname{max}\{ i_1,\dots i_n \} < \operatorname{min} \{j_1,\dots j_n\} \]
so in particular
\[ P_2((a_{i_1},b_{j_1}), (a_{i_2},b_{j_2})) \iff ~ \operatorname{max}\{ i_1, i_2 \} < \operatorname{min} \{j_1, j_2\} \]
\end{enumerate}

Let $\eta \in \omega^{<\omega}$ be given and suppose that either $c_\eta$ has been defined or $\eta = \emptyset$. If $c_\eta$ has been defined,
it will be $(a_i, b_j)$ for some $i<j \in \mathbb{Q}$. Let $\langle k_\ell : \ell < \omega \rangle$ be any $\omega$-indexed subset of
$(i,j) \cap \mathbb{Q}$, or of $\mathbb{Q}$ if $\eta = \emptyset$. Define $c_{\eta^\smallfrown \ell} = (a_{k_\ell}, b_{k_{\ell+1}})$.
Now suppose we have defined the full tree of parameters $c_\eta$ in this way. By fact $(1)$ we see that immediate successors of the same node
are $P_2$-inconsistent. By fact $(2)_n$, paths are consistent, while by fact $(2)_2$, any two $^*$incomparable (Definition \ref{TP1-TP2})
elements $c_\nu, c_\eta$ are $P_2$-inconsistent. 
\end{proof}

\begin{rmk}
The compatible order property in $P_\infty$ is in fact enough to imply maximality in the Keisler order \cite{mm-thesis}.
\end{rmk}

\section{Localization and persistence}
\label{section:l-p}

The goal of these methods is to analyze $\vp$-types, and thus to
concentrate on the combinatorial structure which is ``close to'' or ``inseparable from''
the complete graph $A$ representing a consistent partial $\vp$-type under analysis. Localization and
persistence, defined in this section, are tools for honing in on this essential structure. 

$P_n$ asks about incidence relations on a set of parameters; it will be useful to definably restrict the witness and parameter
sets. For instance:

\begin{itemize}
\item we may ask that the witnesses lie inside certain instances of $\vp$, e.g. by setting 
$P^\prime_1(y) = \exists x(\vp(x;y) \land \vp(x;a))$, i.e. $P^\prime_1 = P_2(y,a)$.
\item we may ask that the parameters be consistent 1-point extensions
(in the sense of some $P_n$) of certain finite graphs $C$. 
For instance, we might define $P^{\prime\prime}_1(y) = P_1(y) \land P_2(y,c_1) \land P_3(y, c_2, c_2)$. 
\end{itemize}

The next definition gives the general form. 

\begin{defn} \emph{(Localization)} \label{localization}
Fix a characteristic sequence $(T, \vp) \rightarrow \langle P_n \rangle$, and choose $B, A \subset M \models T$ with $A$ 
a positive base set and $A=\emptyset$ possible.
\begin{enumerate}
\item \emph{(the localized predicate $P^f_n$)}
A \emph{localization} $P^f_n$ of the predicate $P_n(y_1,\dots y_n)$ around the positive base set $A$ with parameters from $B$ is given by
a finite sequence of triples $f: m \rightarrow \omega \times \fss(y_1,\dots y_n) \times \fss(B)$
where $m<\omega$ and:

\begin{itemize}
\item writing $f(i) = (r_i, \sigma_i, \beta_i)$ and $\check{s}$ for the elements of the set $s$, we have:
\[  P^f_n(y_1,\dots y_n) :=  \bigwedge_{i\leq m} ~~ P_{r_i}(\check{\sigma_i}, \check{\beta_i})       \]
\item for each $\ell < \omega$, $T_1$ implies that there exists a $P_\ell$-complete graph $C_\ell$ such that 
$P^f_n$ holds on all $n$-tuples from $C_\ell$. 
If this last condition does not hold, $P^f_n$ is a \emph{trivial localization}. By \emph{localization} we will always mean
non-trivial localization.
\item In any model of $T_1$ containing $A$ and $B$, $P^f_n$ holds on all $n$-tuples from $A$. 
\end{itemize}

\br
\noindent Write $\loc^B_n(A)$ for the set of localizations of $P_n$ around $A$ with parameters from $B$ (i.e. nontrivial localizations, 
even when $A = \emptyset$). 

\br
\item \emph{(the localized formula $\vp^f$)}
For each localization $P^f_n$ of some predicate $P_n$ in the characteristic sequence of $\vp$, define the corresponding formula
\[\vp^f_n(x;y_1,\dots y_n) := \vp_n(x;y_1,\dots y_n) \land P^f_n(y_1,\dots y_n) \]
When $n=1$, write $\vp^f = \vp^f_1$. 
Let $S^f_\vp(N)$ denote the set of types $p \in S_\vp(N)$ such that for all
$\{ \vp_{i_1}(x;c_{i_1}), \dots \vp_{i_n}(x;c_{i_n}) \} \subset p$, $P^f_n(c_{i_1},\dots c_{i_n})$.
Then there is a natural correspondence between the sets of types 
\[ S^f_\vp(N) \leftrightarrow S_{\vp^f}(N)   \]

\br
\item \emph{(the $^*$localized formula $\vp^{f+\overline{a}}$)}
We have thus far described localizations of the parameters of $\vp$. We will also want to consider restrictions of the possible witnesses
to $\vp$ by adjoining instances of $\vp_k$. That is, set 
\[ \vp^{f+\overline{a}}(x;y) = \vp^{f+a_1,\dots a_k}(x;y) :=  \vp(x;y) \land P^f_1(y) \land \vp_{k}(x;a_1,\dots a_k) \]
where, as indicated, $k = \ell(\overline{a})$. The $^*$ is to emphasize that this is really the construction from $\vp$ of a new, 
though related, formula, which will have its own characteristic sequence, given by: 

\br
\item \emph{(the $^*$localized characteristic sequence $\langle P^{f+\overline{a}}_n : n<\omega \rangle$)}
The sequence $\langle P^{f+\overline{a}}_n : n<\omega \rangle$ associated to the formula $\vp^{f+\overline{a}}$ is given by, for each $n<\omega$,
\[  P^{f+\overline{a}}_n(y_1,\dots y_n) = \bigwedge_{i\leq n} P^f_1(y_i) \land P_{n+k}(y_1,\dots y_n, a_1,\dots a_k)\]
When $f$ or $\overline{a}$ are empty, we will omit them. 
\end{enumerate}
\end{defn}

\begin{rmk} \label{monster-ok}
Convention \ref{localization-depends-on-T} applies: that is, localization is not essentially dependent on the choice of model $M$. See Definition \ref{persistence} (Persistence) and the observation following. 
\end{rmk}

As a first example of the utility of localization, notice that when $\vp$ is simple we can localize to avoid infinite empty graphs. 

\begin{obs} \label{dividing}
Fix a positive base set $A$ for the formula $\psi$, possibly empty. 
When $\psi$ does not have the tree property, then for each $k<\omega$ there is a finite set $C$ over which $\psi$ is not $k$-dividable. 
As a consequence, if $\psi$ does not have the tree property, then for each predicate $P_n$ there is a localization around $A$ 
on which there is a uniform finite bound on the size of a $P_n$-empty graph. 
We can clearly also choose the localizations so that none of $\psi_1,\dots \psi_\ell$ are $k$-dividable 
for any finite $k,\ell$ fixed in advance. 
\end{obs}

\begin{proof}
This is the proof that $D(x=x,\psi,k)<\omega$ for any simple formula $\psi$; see for instance \cite{wagner}.
(For a more direct argument, see the proof of Conclusion \ref{ps-stable} below.)
\end{proof}

The following important property of formulas was isolated by Buechler \cite{buechler}.

\begin{defn} \label{defn-low}
The formula $\vp$ is \emph{low} if there exists $k<\omega$ such that for every instance $\vp(x;a)$ of $\vp$,
$\vp(x;a)$ divides iff it $\leq k$-divides. 
\end{defn}

\begin{obs} \label{stable-low}
If $\vp$ does not have the independence property then $\vp$ is low. 
\end{obs}

\begin{proof}
To show that any non-low formula $\vp$ has the independence property, it suffices to establish the consistency of the following schema.
For $k<\omega$, $\Psi_k$ says that there there exist $y_1,\dots y_{2k}$ such that for every $\sigma \subset 2k$, $|\sigma| = k$, 
\[ \exists x \left( \vrt \vp(x;y_i) \iff i \in \sigma \right) \]
But $\Psi_k$ will be true on any subset of size $2k$ of an indiscernible sequence on which $\vp$ is $k$-consistent but
$(k+1)$-inconsistent, and such sequences exist for arbitrarily large $k$ by hypothesis of non-lowness. 
\end{proof}

\begin{cor} \label{simple-and-low}
When the formula $\psi$ of Observation \ref{dividing} is simple \emph{and} low, 
we can find a localization in which $\psi$ is not $k$-dividable, for any $k$. 
\end{cor}

\subsection{Stability in the parameter space}

The classification-theoretic complexity of the formulas $P_n$ is often strictly
less than that of the original theory $T$. Note that the results here refer to the \emph{formulas} $P_n(y_1,\dots y_n)$, not
necessarily to their full theory $T_1$.  

\begin{obs} \label{op}
Suppose $(T,\vp) \mapsto \langle P_n \rangle$. 
If $P_2(x;y)$ has the order property then $\vp(x;y) \land \vp(x;z)$ is 2-dividable. 
\end{obs}

\begin{proof}
Let $\langle a_i,b_i : i<\omega \rangle$ be a sequence witnessing the order property for $P_2$, so
$P_2(a_i, b_j)$ iff $i<j$.  
This means that $\exists x (\vp(x;a_i) \land \vp(x;b_j))$ iff $i<j$. 
So $\vp(x;a_i) \land \vp(x;b_{i+1})$ are consistent for each $i$, but the set
$\{ \vp(x;a_i) \land \vp(x;b_{i+1}) : i < \omega \}$ is 2-inconsistent. 
\end{proof}

\begin{rmk}
By compactness, without loss of generality the sequence of Observation \ref{op} can be chosen to be $(T-)$indiscernible, and so 
actually witnesses the dividing of some instance of $\vp_2$.

Note that the converse of Observation \ref{op} fails: for $\vp(x;y) \land \vp(x;z)$ to divide it is sufficient
to have a disjoint sequence of ``matchsticks'' in $P_2$ (i.e. $(a_i, b_i) : i<\omega$ such that $P_2(a_i,b_j)$ iff $i=j$), without
the additional consistency which the order property provides. 
\end{rmk}

Nonetheless, work relating the characteristic sequence to Szemer\'edi regularity illuminates the role of the order \cite{mm-thesis}. 

\begin{obs} \label{order-dividable}
Suppose that $(T,\vp) \mapsto \langle P_n \rangle$, and for some $n, k$ and some partition of
$y_1,\dots y_n$ into $k$ object and ($n$-$k$) parameter variables, 
$P_n(y_1,\dots y_k ; y_{k+1},\dots y_n)$ has the order property. 
Then $\vp_n(x ; y_1, \dots y_n)$ is 2-dividable. 
\end{obs}

\begin{proof}
The proof is analogous to that of Observation \ref{op}, replacing the $a_i$ by $k$-tuples and the $b_j$ by
($n$-$k$)-tuples. 
\end{proof}

Thus in cases where we can localize to avoid dividing of $\vp$, we can assume any initial segment of
the associated predicates $P_n$ are stable:

\begin{concl} \label{ps-stable}
For each formula $\vp$ and for all $m<\omega$, if $\vp_{2n}$ does not have the tree property,
then for each positive base set $A$ there are a finite $B$ and
$P^f_1 \in \loc^B_1(A)$ over which $P_2,\dots P_n$ do not have the order property. 
In particular, this holds if $T$ is simple.
\end{concl}

\begin{proof}
We proceed by asking:
do there exist elements $\langle y_i z_i : i<\omega \rangle$ such that (1) each $y_i z_i$ is a $2$-point extension of $A$ 
and (2) $\langle y_i z_i : i<\omega \rangle$ witnesses the order property for $P_2$? 
If not, localize using the finite set of conditions in (1) which prevent (2). Otherwise, let $a_1, b_1$ be the first pair in any
such sequence, set $A_1 := A\cup \{a_1, b_1 \}$ and repeat the argument using $A_1$ in place of $A$. By simplicity, there is a uniform
finite bound on the number of times $\vp_n$ (see Observation \ref{order-dividable}) can sequentially divide. Condition (1)
ensures that the dividing is sequential, corresponding to choosing progressive forking extensions of the partial type corresponding
to $A$. At some finite stage $t$ this will stop, meaning that (1) and (2) fail with $A_t$ in place of $A$; the finite fragment
of (1) which prevents (2) gives the desired localization.
\end{proof}

By way of motivating the next subsection, let us prove the contrapositive: If the order property in $P_2$ persists under repeated localization, 
then $\vp$ has the tree property. Compare the proof of Observation \ref{o-tree} above. Without the compatible order property, we cannot ensure the tree is strict. While that argument built a tree out of a set of parameters which were given all at once (a so-called ``static'' argument), 
the following ``dynamic'' argument must constantly localize to find subsequent parameters, so cannot ensure that elements in different
localizations are inconsistent.  

\begin{lemma}
Suppose that in every localization of $P_1$ (around $A = \emptyset$), $P_2$ has the order property. Then $\vp_2$ has the tree property. 
\end{lemma}

\begin{proof}
Let us describe a tree with nodes $(c_\eta,d_\eta)$, ($\eta \in \omega^{<\omega}$), such that:
\begin{enumerate}
\item for each $\rho \in \omega^\omega$, $\{ c_\eta, d_\eta : \eta \subseteq \rho \}$ is a complete $P_\infty$-graph,
where $\subseteq$ means initial segment.
\item for any $\nu \in \omega^{<\omega}$, $P_2(c_{\eta^\smallfrown i}, d_{\eta^\smallfrown j}) \iff i \leq j$.    
\end{enumerate}

For the base case $(\eta \in \omega^1)$, let $\langle c_i, d_i : i \in \omega \rangle$ be an indiscernible sequence witnessing the order property
(so $P_2(c_i, d_j) \iff i\leq j$) and assign the pair $(c_i, d_i)$ to node $i$. 

For the inductive step, suppose we have defined $(c_\eta, d_\eta)$ for $\eta \in \omega^n$. 
Write $E_\eta = \{ (c_\nu, d_\nu) : \nu \leq \eta  \}$ for the parameters used along the branch to $(c_\eta, d_\eta)$.
Using $\check{x}$ to mean the elements of the set $x$,
let $P^{f_\eta}_1$ be given by $P_{n+1}((y,z), \check{E_\eta})$. 
Let $\langle a_j, b_j : j \in \omega \rangle$ be an indiscernible sequence witnessing the order property inside this localization, 
and define $(c_{\eta^\smallfrown i}, d_{\eta^\smallfrown i}) := (a_j, b_j)$. 

Finally, let us check that this tree of parameters witnesses the tree property for $\vp_2$. On one hand, the order property in $P_2$
ensures that for each $n \in \omega^{<\omega}$, the set
\[ \{ \svrt  \vp_2(x;c_{\eta^\smallfrown i}, d_{\eta^\smallfrown i}) : i \in \omega   \} \]

\noindent is 1-consistent but 2-inconsistent. On the other hand, the way we constructed each localization $P^{f_\eta}_1$ ensured
that each path was a complete $P_\infty$-graph, thus naturally a complete $P^\prime_\infty$-graph, where $\langle P^\prime_n \rangle$ 
is the characteristic sequence of the conjunction $\vp_2$. 
\end{proof}

\begin{rmk}
Example \ref{ex-tree} shows that the condition that $\vp$ has the tree property is necessary, 
but not sufficient, for the order property in $P_2$ to be \emph{persistent}, Definition \ref{persistence} below. 
\end{rmk}

\begin{qst}
Is $SOP_2$ sufficient?
\end{qst}

Compare the issue of whether $SOP_2 \implies SOP_3$: see \cite{DzSh}, \cite{ShUs}. 

\subsection{Persistence}

Localization, Definition \ref{localization} above, gives rise to a natural limit question: what happens when certain $T_0$-configurations persist under all finite localizations?

\begin{defn} \emph{(Persistence)} \label{persistence}
Fix $(T, \vp) \mapsto \langle P_n \rangle$, $M \models T$ sufficiently saturated, and a positive base set $A$, possibly $\emptyset$. 
Let $X$ be a $T_0$-configuration, possibly infinite. Then $X$ is \emph{persistent} around the positive base set $A$
if for all finite $B \subset M$ and for all $P^f_1 \in Loc^B_1(A)$, $P^B_1$ contains witnesses for $X$. 

We will write $X$ is $A$-persistent to indicate that $X$ is persistent around $A$. 
\end{defn}

\begin{note} \label{persistence-asks}
\emph{Persistence asks whether all finite localizations around $A$ contain witnesses for some $T_0$-configuration $X$. The predicates $P_n$
mentioned in $X$ are, however, not the localized versions. We have simply restricted the set from which witnesses can be drawn. 
This is an obvious but important point: for instance, in the proof of Lemma \ref{stable-2} below it is important that the sequence of
$P_2$-inconsistent pairs found inside of successive localizations $P^{f_n}_1$ are $P_2$-inconsistent in the sense of $T_1$.} 
\end{note}

\begin{obs} \emph{(Persistence is a property of the theory $T$)}
The following are equivalent, fixing $T, \vp, \langle P_n \rangle$, $A \subset \mathcal{M}$ a small positive base set
in the monster model, and a $T_0$-configuration $X$. Write $P^f_1(M)$ for the set which $P^f_1$ defines in the model $M$. 
\begin{enumerate}
\item In some sufficiently saturated model $M \models T_1$ which contains $A$, $X$ is persistent around $A$ in $M$. That is,
for every finite $B \subset M$ and every localization $P^f_1 \in \loc^B_1(A)$, there exist witnesses to $X$ in $P^f_1(M)$. 
\item In every model $N \models T_1$, $N \supset A$, for every finite $B \subset N$, every localization $P^f_1 \in Loc^B_1(A)$,
and every finite fragment $X_0$ of $X$, $P^f_1(N)$ contains witnesses for $X_0$.
\end{enumerate}
\end{obs}

\begin{proof}
(2) $\implies$ (1) Compactness.

(1) $\implies$ (2) Suppose not, letting $P^f_1 \in Loc^B_1(A)$ and $X_0$ witness this. To this $P^f_1$ 
we can associate a $T_1$-type $p(y_1,\dots y_{|B|}) \in S(A)$ which says that the localization given by $f$ with 
parameters $y_1,\dots y_{|B|}$ contains $A$ but implies that $X_0$ is inconsistent. 
But any sufficiently saturated model containing $A$ will realize this type, and thus contain such a localization. 
\end{proof}

To reiterate Convention \ref{localization-depends-on-T}, then, we may, as a way of speaking, call a configuration ``persistent'' while working
in some fixed sufficiently saturated model, but we always refer to the corresponding property of $T$. 

\begin{cor}
Persistence around the positive base set $A$ remains a property of $T$ in the language with constants for $A$. 
\end{cor}

Finally, let us check the (easy) fact that persistence of some $T_0$-configuration around $\emptyset$ in some given sequence
$\langle P_n \rangle$ implies its persistence around any positive base set $A$ for that sequence.
Recall that all localizations are, by definition, non-trivial. 

\begin{fact} \label{one-pt-extns}
Suppose that $X$ is an $\emptyset$-persistent $T_0$-configuration in the characteristic sequence $\langle P_n \rangle$ 
and $A$ is a positive base set for $\langle P_n \rangle$. Then $X$ remains persistent around $A$.
\end{fact}

\begin{proof}
Let $p(x_0,\dots )$ in the language $\mathcal{L}(=,P_1,P_2,\dots)$ describe the type, in $V_X$-many variables, of the configuration $X = (V_X, E_X)$.
Let $q(y) \in S(A)$ be the type of a 1-point $P_\infty$-extension of $A$ in the language $\mathcal{L}_0 = \{ P_n : n<\omega \} \cup \{=\}$. 
We would like to know that
$q(x_0), q(x_1), \dots, p(x_0, \dots)$ is consistent, i.e., that we can find, in some given localization,
witnesses for $X$ from among the elements which consistently extend $A$. If not, for some finite subset
$A^\prime \subset A$, some $n<\omega$, and some finite fragments $q^\prime$ of $q|_{A^\prime}$ and
$p^\prime$ of $p$,
\[ q^\prime(x_0) \cup \dots \cup q^\prime(x_n)  \vdash \neg p^\prime(x_0,\dots x_n) \]

\noindent But now localizing $P_1$ according to the conditions on the lefthand side (which are all positive
conditions involving the $P_n$ and finitely many parameters $A^\prime$) shows that $X$ is not persistent,
contradiction.
\end{proof}

\section{Dividing lines: Stability, Simplicity, NIP}
\label{dividing-lines}

The first natural question for persistence is: given $n$, when isn't it possible to localize $P_1$ so that $P_n$ is a complete graph?
The answer, surprisingly, is: in the presence of the independence property. 
This section gives the argument, using the language of persistence to give a new description of NIP and of simplicity,
Theorem \ref{char-stable} and Theorem \ref{char-simple} below. 
Recall that a theory $T$ is NIP \cite{Sh715} if no formula of $T$ has the independence property; for more on this
hypothesis, see \cite{Usv}. 

\subsection{NIP: the case of $P_2$} 
We will see that if $\vp$ is NIP then we can localize around any fixed positive base set so that
$P_2$ is a complete graph.

The argument in this technically simpler case will generalize
without too much difficulty. We first revisit an avatar of the independence property. 

\begin{defn} \label{omega-2-arrays} \emph{($(\omega, 2)$-arrays revisited)}
\begin{enumerate}
\item The predicate $P_n$ is $(\omega, 2)$ if there is $C := \{ a^t_i: t<2, i<\omega \}$ 
such that for all $\ell \leq n$, any $\ell$-element subset $C_0$,
\[ P_\ell(C_0) ~~\iff~~ \left( a^t_i, a^s_j \in C_0 \implies (i\neq j) \lor (t=s) \right) \]
\item If for all $n<\omega$, $P_n$ is $(\omega,2)$, we say that $P_\infty$ is $(\omega, 2)$. 

\item A \emph{path} through the $(\omega,2)$-array $A$ is a set $X \subset A$ which contains no more than one element from each column.
So paths are positive base sets.  
\end{enumerate}
\end{defn}

\begin{rmk} \label{omega-2}
If $P_\infty$ is $(\omega, 2)$, then $\vp$ has the independence property.
\end{rmk}

\begin{proof}
Let $X$ be a maximal path through the $(\omega, 2)$-array $A$. Choose any $\sigma, \tau \subset X$ finite and disjoint.
Let $Y_{\sigma, \tau} \subset A$ be a maximal path such that $\sigma \subset Y_{\sigma,\tau}$ and $Y_{\sigma,\tau} \cap \tau = \emptyset$.
$Y_{\sigma,\tau}$ is a positive base set, so any element $c$ realizing the corresponding $\vp$-type will satisfy
$a \in \sigma \rightarrow \vp(c;a)$ and $b \in \tau \rightarrow \neg \vp(c;b)$. Thus $\vp$ has the independence property on $X$.
\end{proof}

\begin{lemma} \label{springboard-2} (Springboard lemma for $2$)
If $\vp$ is stable then there is a finite localization $P^f_1$ for which TFAE:
\begin{enumerate} 
\item There exists $X \subset P^f_1$, $X$ an $(\omega,2)$-array wrt $P_2$
\item There exists $Y \subset P^f_1$, $Y$ an $(\omega,2)$-array wrt $P_\infty$
\end{enumerate}  
\end{lemma}

\begin{proof}
Choose the localization $P^f_1$ according to Observation \ref{dividing} so that neither $\vp$ nor $\vp_2$ are dividable using parameters from
$P^f_1$. This is possible because stable formulas are simple and low, and $\vp$ stable implies $\vp_2$ stable. 
Let $Z = \langle c^t_i : t<2, i<\omega \rangle \subset P^f_1$ be an indiscernible sequence of pairs which is an $(\omega, 2)$-array
for $P_2$.
Each of the sub-sequences $\langle c^0_i : i<\omega \rangle$, $\langle c^1_i : i<\omega \rangle$ is indiscernible, 
so will be either $P_2$-complete or $P_2$-empty;
by choice of $P^f_1$, they cannot be empty.

It remains to show that any path $X \subset Z$ is a $P_\infty$-complete graph. 
Suppose not, and let $n$ be minimal so that the $n$-type of some increasing sequence of elements $z^{t_1}_1,\dots z^{t_n}_n$  
implies $\neg \exists x (\bigwedge_{i<n} \vp(x;z^{t_i}_i))$. Choose an infinite indiscernible 
subsequence of pairs $Z^\prime \subset Z^2$ of the form $\langle c^0_i, c^1_{i+1} : i \in W \subset \omega \rangle$. Then the set
$\{ \vp(x;c^0_i) \land \vp(x;c^1_{i+1}) : i \in W \}$ will be 1-consistent by definition but $n$-inconsistent by assumption
(though not necessarily sharply $n$-inconsistent). This contradicts the assumption that $\vp_2$ is not dividable in $P^f_1$. 
\end{proof}

When the formula is low but not necessarily simple, bootstrapping up to $P_\infty$ is still possible but requires 
a stronger initial assumption on the array. 

\begin{cor} \label{springboard-2-low}
Suppose the formulas $\vp$ and $\vp_2$ are low. Then there exists
$k<\omega$ such that, in any localization $P^f_1$, TFAE:
\begin{enumerate} 
\item There exists $X \subset P^f_1$, $X$ an $(\omega,2)$-array wrt $P_k$
\item There exists $Y \subset P^f_1$, $Y$ an $(\omega,2)$-array wrt $P_\infty$
\end{enumerate} 
\end{cor}

\begin{proof}
Let $k_0$ be a uniform finite bound on the arity of dividing of instances of $\vp$ and $\vp_2$, using lowness; 
by the proof of the previous Lemma, any $k > 2k_0$ will do.
\end{proof}

Recall from Definition \ref{notation} that an ``empty pair'' is the $T_0$-configuration given by $V_x=2, E_x=\{ \{1\},\{2\} \}$,
i.e., a pair y, z such that $P_1(y)$, $P_1(z)$ but $\neg P_2(y,z)$.

\begin{lemma} \label{array-for-2}
Suppose $\vp$ is stable, and that every localization $P^f_1$ around some fixed 
positive base set $A$ contains an empty pair. 
Then $P_\infty$ is $(\omega, 2)$. 
\end{lemma}

\begin{proof}
Choose $P^{f_0}_1$ to be a localization given by Lemma \ref{springboard-2}.
We construct an $(\omega, 2)$-array as follows.  

At stage $0$, let $c^0_0, c^1_0$ be any pair of $P_2$-incompatible elements each of which is a consistent $1$-point extension of $A$ in $P^{f_0}_1$.
At stage $n+1$, write $C_n$ for $\{ c^t_i : t<2,i\leq n\}$ and 
suppose we have defined $P^{f_{n}}_1 \in \loc^{C_{n}}_1(A)$. By hypothesis, there are $c^0_{n+1}, c^1_{n+1} \in P^{f_{n}}_1$
such that $\neg P_2(c^0_{n+1}, c^1_{n+1})$ and such that each $c^i_{n+1}$ is a consistent 1-point extension of $A$ (Fact \ref{one-pt-extns}). 
Let $C_{n+1} = C_n \cup \{c^0_{n+1}, c^1_{n+1} \}$ and define $P^{f_{n+1}}_1 \in \loc^{C_{n+1}}_1(A)$ by
\[  P^{f_{n+1}}_1(y) = P^f_1(y) \land P_2(y;c^0_{n+1}) \land P_2(y; c^1_{n+1}) \]
Thus we construct an $(\omega,2)$-array for $P_2$, as desired. Applying Lemma \ref{springboard-2} we obtain an
$(\omega, 2)$-array for $P_\infty$. 
\end{proof}

\begin{concl} \label{stable-2}
Suppose that $\vp$ is stable, $(T, \vp) \mapsto \langle P_n \rangle$ and $A$ is a positive base set. Then empty pairs are not persistent around $A$.
\end{concl}

\begin{proof}
By stability, we may work inside the localization given by Lemma \ref{springboard-2}. 
Suppose empty pairs were persistent around $A$. By Lemma \ref{array-for-2}, $P_\infty$ is $(\omega,2)$, which by Remark \ref{omega-2} implies that $\vp$ has the independence property: contradiction. 
\end{proof}

In order to replace the hypothesis of stable with low, we will need to replace $P_2$-consistency in the proof of Lemma \ref{array-for-2}
with $P_k$-consistency. This argument is given in full generality in Lemma \ref{stable-k}, but here we state the result: 

\begin{cor} \label{low-2} (to Corollary \ref{springboard-2-low})
Suppose $\vp$ and $\vp_2$ are both low.
Suppose every localization $P^f_1$ around some fixed positive base set $A$ contains an empty pair. 
Then $P_\infty$ is $(\omega, 2)$. 
\end{cor}

In fact, modulo the proof of Lemma \ref{stable-k} we have shown:

\begin{concl} \label{NIP-for-2}
Suppose that $\vp$ is NIP, $(T, \vp) \mapsto \langle P_n \rangle$ and $A$ is a positive base set. 
Then empty pairs are not persistent around $A$.
\end{concl}

\begin{proof}
By Observation \ref{stable-low} all NIP formulas are low. By Fact \ref{then-vp-has-ip}, $\vp$ NIP implies $\vp_2$ is NIP and therefore low.
By Corollary \ref{low-2} and Remark \ref{omega-2}, the persistence of empty pairs would imply $\vp$ has the independence property, 
contradiction.
\end{proof}

\subsection{NIP: the case of $n$} 

We now build a more general framework, working towards Theorem \ref{stable-is-complete}, which generalizes Conclusions \ref{stable-2}-\ref{NIP-for-2} to the case of arbitrary $n < \omega$: if $T$ is NIP then no $P_n$-empty tuple can be persistent. The basic strategy is as follows. 
If a $P_n$-empty tuple is persistent, Lemma \ref{stable-k} produces an $(\omega, n)$-array.
In this higher-dimensional case, in order to extract the independence property from an $(\omega, n)$-array via Observation \ref{omega-k-ip},
we need the array to have an additional property called sharpness.  
The ``sharpness lemma,'' Lemma \ref{sharpness-lemma}, returns an array of the correct form at the cost of possibly adding finitely many parameters.
Fact \ref{then-vp-has-ip} then pulls this down to the independence property for $\vp$. 

With some care, we are able to get quite strong control on the kind of localization used. When $T$ is stable in addition to NIP,
the argument can be done with a uniform finite bound (as a function of $n$) on the arity of the predicates $P_m$ used in localization.

\begin{defn} \label{omega-n-arrays} \emph{($(\omega,n)$-arrays revisited)}  Assume $n \leq r < \omega$. 
Compare Definition \ref{diagrams-arrays};
here, the possible ambiguity of the amount of consistency will be important. 

\begin{enumerate} 
\item The predicate $P_r$ is $(\omega, n)$ if there is $C = \{ c^t_i: t<n, i<\omega \} \subset P_1$ such that,
for all $c^{t_1}_{i_1}, \dots c^{t_r}_{i_r} \in C$,
\begin{itemize}
\item $r$-tuples from $r$ distinct columns are consistent, i.e.
\[ \bigwedge_{j,k\leq r} i_j \neq i_k ~ \implies P_r(c^{t_1}_{i_1}, \dots c^{t_r}_{i_r}) \]
\item and no column is entirely consistent, i.e. for all $\sigma \subset r$, $|\sigma| = n$, 
\[ \bigwedge_{j,k \in \sigma} i_j = i_k ~ \implies ~\neg P_r(c^{t_1}_{i_1}, \dots c^{t_r}_{i_r}) \]
\end{itemize}
\noindent Any such $C$ is an \emph{$(\omega, n)$-array}. The precise arity of consistency is not specified, see condition \emph{(4)}. 
\item If for all $n \leq r<\omega$, $P_r$ is $(\omega,n)$, say that \emph{$P_\infty$ is $(\omega, n)$}. 
\item A \emph{path} through the $(\omega,n)$ array $C$ is a set $X \subset C$ which contains no more than $n$-$1$ elements from each column. 
\item $P_r$ is \emph{sharply} $(\omega, n)$ if it contains an $(\omega,n)$-array $C$ on which, moreover, for all 
$\{c^{t_1}_{i_1}, \dots c^{t_r}_{i_r} \} \subset C$
\[ P_r(c^{t_1}_{i_1}, \dots c^{t_r}_{i_r}) ~\iff~ \bigwedge_{\sigma \subset r, |\sigma| = n} 
\left( \bigwedge_{j,k \in \sigma} i_j = i_k ~ \implies ~\bigvee_{j \neq k \in \sigma} t_j = t_k \right) \]
i.e., if every path is a $P_r$-complete graph. 
\item $P_\infty$ is sharply $(\omega, n)$ if $P_r$ is sharply $(\omega, n)$ for all $n \leq r < \omega$.
\end{enumerate}
\end{defn}

\begin{rmk} \label{nontrivial-descent}
\begin{enumerate}
\item Every $(\omega, 2)$-array is automatically sharp. 
\item Suppose $P_\infty$ has an $(\omega, n)$-array; this does not necessarily imply that $P_\infty$ has an $(\omega, m)$-array for $m<n$, because
$m$ elements from a single column need not be inconsistent, e.g. if the $(\omega, n)$-array is sharp.  
\end{enumerate}
\end{rmk}

\begin{obs} \label{omega-k-ip}
If $P_\infty$ is sharply $(\omega, k)$ then $\vp_{k-1}$ has the independence property. 
\end{obs}

\begin{proof}
Let $X = \langle a^1_i,\dots a^{k}_i : i<\omega \rangle$ be the array in question; then $\vp_{k-1}$ has the independence property on any
maximal path, e.g. 
$B := \langle a^1_i,\dots a^{k-1}_i : i<\omega \rangle$. To see this, fix any $\sigma, \tau \subset \omega$ finite disjoint; then by the
sharpness hypothesis $\{ a^1_i,\dots a^{k-1}_i : i \in \sigma \} \cup \{ a^2_j,\dots a^{k}_j : j \in \tau \}$ is a $P_\infty$-complete graph and thus 
corresponds to a consistent partial $\vp$-type $q$. But any realization $\alpha$ of $q$ cannot satisfy $\vp(x;a^1_j)$ 
for any $j \in \tau$, because $P_{k}$ does not hold on the columns. A fortiori $\neg \vp_k(\alpha;a^0_j,\dots a^{k-1}_j)$. 
\end{proof}

Let us write down some conventions for describing types in an array.

\begin{defn} \label{array-types} Let $x^t_i, x^s_j$ be elements of some $(\omega, n)$-array $X$.
\begin{enumerate}
\item Let $[x^t_i] = \{ x^s_j \in X : j = i\}$, i.e. the elements in the same column as $x^t_i$. 
\item Let $X_0 = \{ x^{t_1}_{i_1}, \dots x^{t_\ell}_{i_\ell} \} \subset X$ be a finite subset. The
\emph{column count} of $\{ x^{t_1}_{i_1}, \dots x^{t_\ell}_{i_\ell} \}$ is the unique tuple
$(m_1,\dots m_\ell) \in \omega^\ell$ such that:
\begin{itemize}
\item $m_i \geq m_{i+1}$ for each $i \leq \ell$
\item $\Sigma_i ~m_i = \ell$
\item if $Y_0 = \{ y_1,\dots y_r \}$ is a maximal subset of $X_0$
such that $y,z \in Y_0, y \neq z ~ \rightarrow~ y \notin [z]$, then some permutation of 
\[   \left(\left| \vrt [y_1] \cap X_0\right|, \dots , \left| \vrt [y_r] \cap X_0\right|\right) \]
is equal to $(m_1,\dots m_\ell)$. 
\end{itemize} 
\noindent In other words, we count how many elements have been assigned to each column, and put these counts in descending order of size.
Write $\operatorname{col-ct}(\overline{x})$ for this tuple. 
\item  
Let $\leq$ be the lexicographic order on column counts, i.e. $(1, 1, \dots ) < (2,1,\dots)$. This is a discrete linear order, so we can define
$(m_1,\dots m_\ell)^+ $ to be the immediate successor of $(m_1, \dots m_\ell)$ in this order. 
Define $\operatorname{gap}((m_1,\dots m_\ell)) = m_i$ where $((n_1,\dots n_\ell)^+ = (m_1,\dots m_\ell)$ and $\forall j\neq i$ $m_j = n_j$,
i.e. the value which has just incremented.
\end{enumerate}
\end{defn}

By analogy to Lemma \ref{springboard-2} and its corollary,

\begin{lemma} \emph{(Springboard lemma)} \label{springboard-k}
Fix $2 \leq n<\omega$, and let $\langle P_n \rangle$ be the characteristic sequence
of $(T,\vp)$.  
Suppose that the formulas $\vp, \vp_2, \dots \vp_{2n-2}$ are low. 
Then there exist $1 \leq k_0 < \omega$ and a localization $P^f_1$ of $P_1$ in which the following are equivalent:
\begin{enumerate}
\item $P^f_1$ contains a sharp $(\omega, n)$-array for $P_\mu$, where $\mu = (2n-2)k_0$.
\item $P^f_1$ contains a sharp $(\omega, n)$-array for $P_\infty$. 
\end{enumerate}
\end{lemma}

\begin{proof}
Assume (1), so let $C = \{ c^t_i : t<n, i<\omega \} \subset P^f_1$ be sharply $(\omega, n)$ for $P_\mu$, chosen without loss of 
generality to be an indiscernible sequence of $n$-tuples.
Fix a path $Y=y_1,\dots y_{m}$ of minimal size $m > n$ such that $\neg P_{m}(y_1,\dots y_m)$. 
Let $S := \{ c^0_i,\dots c^{n-1}_i, c^1_{i+1},\dots c^n_{i+1} : i<\omega \} \subset C^{2n-2}$
be a sequence of pairs of offset $(n-1)$-tuples. 

Note that $S$ is $1$-consistent as we assumed (1).

On the other hand, $C$ is indiscernible, so any increasing sequence of $m$ elements from $S$ will cover all the possible $m$-types from $C$.
Since $Y$ is inconsistent, this implies that $S$ is $m$-inconsistent. These $m$ elements will be distributed over at least $\frac{m}{2n-2}$
instances of $\vp_{2n-2}$; by inductive hypothesis, one fewer element, thus one fewer instance, would be consistent. 
Thus $\vp_{2n-2}$ is sharply $m^\prime$-dividable for  some $m^\prime \geq k_0$.

The appropriate $k_0$ is thus a strict upper bound on the possible
arity of dividing of each of the formulas $\{ \vp_{2\ell-2} : 1 \leq \ell \leq k_0 \}$, which exists by lowness.
When $T$ is low but possibly unstable, determining $k_0$ is the important step; no localization is then necessary.
When $T$ is stable, however, w.l.o.g. $k_0 = 2n-2$ as by Corollary \ref{simple-and-low}
we can simply choose a localization in which the $2n-2$ formulas are not $k$-dividable for any $k$. 
\end{proof}

We next give a lemma which will extract a sharp array from an array.
Recall that $P^{\overline{a}}_\infty$ is the $^*$localized \emph{sequence} from Definition \ref{localization},
i.e. the characteristic sequence of the formula $\vp(x;y)\land \bigwedge_{a \in \overline{a}} \vp(x;a)$.

\begin{lemma} \label{sharpness-lemma} \emph{(Sharpness lemma)}
Let $\overline{a} \subset P_1$ be finite, $n<\omega$. Suppose that $P^{\overline{a}}_\infty$ contains an $(\omega, n)$-array.
Then there exist $\overline{a}^\prime$, $\ell$ with $\overline{a} \subseteq \overline{a}^\prime \subset P_1$ and $2 \leq \ell \leq n$ such that $P^{\overline{a}^\prime}_\infty$ contains a sharp $(\omega, \ell)$-array.
\end{lemma}

\begin{proof}
Let us show that, given an $(\omega, n)$-array for $P^{\overline{a}}_\infty$, either

\begin{itemize}
\item there is a \emph{sharp} $(\omega, n)$ array for $P^{\overline{a}}_\infty$, or else 
\item by adding no more than finitely many parameters 
we can construct an $(\omega, \ell)$-array for $P^{\overline{a}^\prime}_\infty$ and some $\ell < n$.
\end{itemize}
 
Note that the second is nontrivial by Remark \ref{nontrivial-descent}. 
As an $(\omega, 2)$-array is automatically sharp, we can then iterate the argument to obtain the lemma. 

We have, then, some $(\omega, n)$-array $C$ in hand.
Without loss of generality $C$ is an indiscernible sequence of $n$-tuples. 
If every path through $C$ is a $P^{\overline{a}}_\infty$-complete graph then $C$ is a sharp $(\omega, n)$-array and we are done.
Otherwise, choose some finite $Z \subset C$ whose column count is as small as possible subject to the conditions:

\begin{enumerate}
\item $Z$ is a path
\item There exists some $Y \subset C$ such that
\begin{itemize}
\item[(i)] $Y \cap [Z] = \emptyset$
\item[(ii)] $y_1, y_2 \in Y \implies [y_1] \cap [y_2] = \emptyset$
\end{itemize}
but $Z \cup Y$ is not a $P^{\overline{a}}_\infty$-complete graph. 
\end{enumerate}

\noindent In other words, $Z$ is a possible new parameter set which is just slightly too large: the subset of $C$ which is consistent with $Z$
fails to be an $(\omega, n)$-array because some set of elements from distinct columns is not consistent relative to $Z$.
Set $X := Z \cup Y$, where $Y$ is the finite sequence from (2). 

The assumption that $C$ is not sharp gives an unspecified finite bound on $|Z|$; 
in fact the springboard lemma gives a more informative
bound $k \geq 2n-2$. On the other hand, by definition of $(\omega, n)$-array, any such $Z$ must contain at least two elements from the same column,
so $|Z| > 1$ and we can find our witness working upwards on column count. 
Because $C$ is an indiscernible sequence of $n$-tuples, we may assume that the elements of $Z$ are in 
columns which are infinitely far apart. 
Finally, if $Z_0 \subsetneq Z$, then $\colct{(Z_0)} < \colct{(Z)}$. 
So for any $W \subset C$ satisfying conditions (2).(i)-(ii) just given, $Z_0 \cup W$ is a $P^{\overline{a}}_\infty$-complete graph.

In particular, we can choose a partition $X = X_0 \cup X_1$ where 

\begin{itemize}
\item[(I)] $X_0 \cap X_1 = \emptyset$, $\emptyset \subsetneq X_1 \subset X$
\item[(II)] $x, x^\prime \in X_1 \implies [x] = [x^\prime]$ 
\item[(III)] $n > \ell := |X_1| = gap(\colct(Z)) > 1$
\item[(IV)] For any $W \subset C$ satisfying conditions (2).(i)-(ii), $X_0 \cup W$ is a $P^{\overline{a}}_\infty$-complete graph.
\end{itemize}

To finish, let $a^\prime = a \cup X_0$ and let $C^\prime \subset C$ be an infinite sequence of $\ell$-tuples which
realize the same type as $X_1$ over $a \cup X_0$. (For instance, restrict $C^\prime$ to the rows containing elements of $X_1$
and to infinitely many columns which do not contain elements of $X_0$.) 
Since $Z$ was chosen to be a path, $\ell < n$ (condition (III)) and $|a^\prime| < |X| < \omega$.
By condition (2), $\neg P^{{\overline{a}^\prime}}_\ell (\overline{c})$
for any column $\overline{c}$ of $C^\prime$. On the other hand, by condition (IV) any subset of $C^\prime$ containing no more than one
element from each column is a $P^{\overline{a}^\prime}_\infty$-complete graph. Thus $C^\prime$ is an $(\omega, \ell)$-array
for $P^{\overline{a}^\prime}_\infty$, as desired. If it is not sharp, repeat the argument.
\end{proof}

\begin{fact} \label{then-vp-has-ip} The following are equivalent for a formula $\vp(x;y)$.
\begin{enumerate}
\item $\vp$ has the independence property. 
\item For some $n<\omega$, $\vp_n$ has the independence property. 
\item For every $n<\omega$, $\vp_n$ has the independence property. 
\item Some $^*$localization $\vp^{\overline{a}}$ has the independence property. 
\end{enumerate}
\end{fact}

\begin{proof}
(1) $\rightarrow$ (3) $\rightarrow$ (2) $\rightarrow$ (1) $\rightarrow$ (4) are straightforward: 
use the facts that the formulas $\vp_i$, $\vp_j$ generate the same space of types, 
and that the independence property can be characterized in terms 
of counting types over finite sets (\cite{Sh:c}:II.4). Finally, (4) $\rightarrow$ (2) as
we have simply specified some of the parameters. 
\end{proof}

\begin{lemma} \label{stable-k}
Suppose that for some $n<\omega$, every localization of $P_1$ around some fixed positive base set $A$ contains 
an $n$-tuple on which $P_n$ does not hold. Then $P_\infty$ is $(\omega, n)$, though not necessarily sharply $(\omega, n)$.
\end{lemma}

\begin{proof}
Let us show that $P_k$ is $(\omega, n)$ for any $k \geq n$. This suffices as, by Convention \ref{localization-depends-on-T}, we may apply compactness.

Fix $k \geq n$ and let $P^f_1$ be any localization, for instance that of Lemma \ref{springboard-k}. 

At stage $0$, let $c^0_0, c^1_0, \dots c^{n-1}_0 \subset P^{f_0}_1 := P^f_1$ be an $n$-tuple of elements on which $P_n$
does not hold, chosen by Fact \ref{one-pt-extns} so that each $c^i_0$ is a consistent $1$-point extension [in the sense of $P_n$] of $A$.
Let $X_0 = \{ \{c^i_0 \} : i \leq n \}$ be the set of these singletons. Write $\check{x}$ to denote the elements of $x$.
Define
\[ P^{f_{1}}_1(y) = P^{f_0}_1(y) \land \bigwedge_{x \in X_0} P_2(y;\check{x}) \]
which includes $A$ by construction. 

At stage $m+1$, write $C_m$ for $\{ c^t_i : t<n, i\leq m\}$ and consider 
the localized set of elements $P^{f_{m}}_1 \in \loc^{C_{m}}_1(A)$. Let
\[ X_m := \{ x \subset C_m : ~|x| = k-1 ~\mbox{and for all i $<$ m},~ \left| x \cap \left(C_{i+1} \setminus C_i\right) \right| \leq 1 \} \]
i.e. sets which choose no more than one element from each stage in the construction. 

By hypothesis, there are $c^0_{m+1}, \dots c^{n-1}_{m+1} \in P^{f_{m}}_1$
such that $\neg P_n(c^0_{m+1}, \dots c^{n-1}_{m+1})$ and such that for all $x \in X_m$, 
each $c^i_{m+1}$ is a consistent 1-point extension of $A \cup x$, in the sense of $P_n$. 
Let $C_{m+1} = C_m \cup \{c^0_{m+1}, \dots c^{n-1}_{m+1} \}$, and let $X_{m+1}$ be the
sets from $C_{m+1}$ which choose no more than one element from each stage in the construction. 
We now define $P^{f_{m+1}}_1 \in \loc^{C_{m+1}}_1(A)$ by
\[  P^{f_{m+1}}_1(y) = P^{f_m}_1(y) \land \bigwedge_{x \in X_{m+1}} P_k(y;\check{x}) \]
(If $m<k$, the parameters from $\check{x}$ need not necessarily be distinct.) Again, this localization
contains $A$ by construction. 
Thus we construct an $(\omega,n)$-array for $P_k$, as desired. As $k$ was arbitrary, we finish.
\end{proof}

Recall that a $P_n$-empty tuple is any $T_0$-configuration for which $X = n$, $\{1,\dots n\} \notin E_x$,
i.e. $y_1,\dots y_n \in P_1$ such that $\neg P_n(y_1,\dots y_n)$.
We are now in a position to prove:

\begin{theorem} \label{stable-is-complete}
Suppose that $\vp$ is NIP, $(T, \vp) \mapsto \langle P_n \rangle$ and $A$ is a positive base set for $\vp$.
Then for each $n<\omega$, $P_n$-empty tuples are not persistent around $A$.
\end{theorem}

\begin{proof}
By lowness, we work inside the localization $P^f_1 \supset A$ given by the Springboard Lemma \ref{springboard-k}. 
Suppose that for some $n<\omega$, $P_n$-empty tuples are persistent. Apply Lemma \ref{stable-k}
to obtain an $(\omega, n)$-array for $P_\infty$, which is not necessarily sharp. 
The Sharpness Lemma \ref{sharpness-lemma} then gives a sharp $(\omega, \ell)$-array for $P^{\overline{a}}_\infty$,
where $\overline{a} \subset P_1$ is a finite set of parameters. The sequence $\langle P^{\overline{a}}_n \rangle$
is just the characteristic sequence of the $^*$localized formula $\vp^{\overline{a}}$, that is, $\vp(x;y) \land \vp_m(x;\overline{a})$, where $m = |\overline{a}|$.
By Observation \ref{omega-k-ip} this means $\vp^{\overline{a}}_{\ell-1}$ has the independence property. 
Now by (2) $\rightarrow$ (1) of Fact \ref{then-vp-has-ip} applied to $\vp^{\overline{a}}$, $\vp^{\overline{a}}$ has the independence property.
By (4) $\rightarrow$ (1) of the same fact, $\vp$ must also have the independence property, contradiction. 
\end{proof}

\begin{cor}
Suppose that $\vp$ is $NIP$, $(T, \vp) \mapsto \langle P_n \rangle$ and $A$ is a positive base set for $\vp$. 
Then for each $n <\omega$, there is a localization $P^f_1$ such that:
\begin{enumerate}
\item $A \subset P^f_1$
\item $\{ y_1,\dots y_n \} \subset P^f_1 \implies P_n(y_1,\dots y_n)$
\end{enumerate}
\end{cor}

In other words, when $\vp$ is $NIP$, given a positive base set $A$ there is, for each $n$, a definable restriction of $P_1$ containing
$A$ on which $P_n$ is a complete graph.
In the other direction, if $\vp$ has the independence property then $P_\infty$ is $(\omega, 2)$ by Claim \ref{ip-in-cs}, so in particular
$P_1$ is not a complete graph. In fact: 

\begin{theorem} \label{char-stable} Let $\vp$ be a formula of $T$ and $\langle P_n : n<\omega \rangle$ its characteristic sequence. 

Then the following are equivalent for any positive base set $A$:
\begin{enumerate}
\item There exists a localization $\vp^f$ of $\vp$ such that $\vp^f$ is NIP and $P^f_1 \supset A$.
\item There exists a localization $\vp^g$ of $\vp$ such that $\vp^g$ is stable and $P^g_1 \supset A$.  
\item For every $n<\omega$, there exists a localization $P^{f_n}_1 \supset A$ which is a $P_n$-complete graph.   
\end{enumerate}
\end{theorem}

\begin{proof} Note that $\vp^f =\vp$ and $\vp^g=\vp$ are possible.

(2)$\rightarrow$(1) because stable implies NIP.

(1)$\rightarrow$(3) is just Theorem \ref{stable-is-complete}: no $P_n$-empty tuple is persistent, so eventually one obtains a localization
which is a complete graph.

(3)$\rightarrow$(2) By Claim \ref{diagrams-instability}, if $\vp^g$ has the order property its
associated $P^g_1$ contains a diagram in the sense of Definition \ref{diagrams-arrays}. Thus it contains an empty pair,
and so a fortiori a $P_n$-empty tuple, for each $n$.
\end{proof}

\begin{expl} \label{nip-q-remark}
Consider $(\mathbb{Q}, <)$,
let $\vp(x;y,z) = y > x > z$ and let the positive base set $A$ be given by concentric intervals
$\{ (a_i, b_i) : i <\kappa \} \subset P_1$. Then there is indeed a $P_2$-empty pair $(c_1,c_2), (d_1,d_2)$ 
which are each consistent 1-point extensions of $A$ -- namely, any pair of disjoint intervals lying in the cut
described by the type corresponding to $A$. Localizing to require consistency with any such pair amounts to giving 
a definable complete graph containing $A$, i.e. realizing the type. 
\end{expl}

\subsection{Simplicity}

We have seen that the natural first question for persistence, whether there exist persistent empty tuples,
characterizes stability: Theorem \ref{char-stable}. Here we will show that a natural next question, whether 
there exist persistent infinite empty graphs, characterizes simplicity. Recall that a formula is simple
if it does not have the tree property; see \cite{kim-pillay}, \cite{wagner}. 

Notice that we have an immediate proof of this fact by Observation \ref{dividing}, which appealed to finite $D(\vp, k)$-rank
for simple formulas to conclude that infinite empty graphs are not persistent. 
Let us sketch the framework for a different proof by analogy with the previous section. 
This amounts to deriving Observation \ref{dividing} directly in the characteristic sequence. 

\begin{rmk}
In the case of stability, much of the work came in establishing sharpness of the $(\omega, \ell)$-array.
Here, since the persistent configuration is infinite, we have compactness on our side; we may in fact always choose
the persistent empty graphs to be indiscernible and uniformly $k$-consistent but $(k+1)$-inconsistent, for some 
given $k<\omega$. 
\end{rmk}

\begin{obs} \label{trees-in-P1}
Suppose that $(T, \vp) \mapsto \langle P_n \rangle$. Then the following are equivalent:
\begin{enumerate}
\item there is a set $T = \{ a_\eta : \eta \in 2^{<\omega} \} \subset P_1$
such that, writing $\subseteq$ for initial segment:
\begin{enumerate}
\item For each $\nu \in 2^\omega$, $\{ a_\eta : \eta \subset \nu \}$ is a complete $P_\infty$-graph.
\item For some $k<\omega$, and for all $\rho \in \omega^{<\omega}$, the set
$\{  a_{\rho^\smallfrown i} : i<\omega \} \subset P_1$ is a $P_{k}$-empty graph. 
\end{enumerate}
\item $\vp$ has the $k$-tree property.
\end{enumerate}
\end{obs}

\begin{proof}
This is a direct translation of Definition \ref{TP1-TP2}. 
\end{proof}

\begin{lemma} \label{simple-trees}
Let $X_k$ be the $T_0$-configuration describing a strict $(k+1)$-inconsistent sequence, i.e. $V_{X_k} = \omega$
and $E_{X_k} = \{ \sigma : \sigma \subset \omega, |\sigma| \leq k \}$. 
Suppose that for some fixed $k<\omega$ and some formula $\vp$, $X_k$ is persistent in the characteristic sequence 
$\langle P_n \rangle$ of $\vp$. Then $\vp$ is not simple. 
\end{lemma}

\begin{proof}
Let us show that $\vp$ has the tree property, around some positive base set $A$ if one is specified. At stage $0$,
by hypothesis there exists an infinite indiscernible sharply $(k+1)$-inconsistent 
sequence $Y_0 \subset P_1$, each of whose
elements can be chosen to be a consistent $1$-point extension of $A$ in the sense of $P_\infty$ by Fact \ref{one-pt-extns}. 
Set $a_i$ to be the $i$th element of this sequence, for $i<\omega$.

At stage $t+1$, suppose we have constructed a tree of height $n$, $T_n = \{ a_\eta : \eta \in \omega^{\leq n} \}$
such that, writing $\subseteq$ for initial segment:
\begin{itemize}
\item every path is a consistent $n$-point extension of $A$, i.e.
$A \cup \{ a_\eta : \eta \subseteq \nu \}$ is a complete $P_\infty$-graph, for each $\nu \in \omega^n$;
\item for all $0 \leq k < n$ and all $\eta \in \omega^k$, $\{ a_{\eta^\smallfrown i}  : i<\omega \}$ is
$P_k$-complete but $P_{k+1}$-empty.
\end{itemize}

We would like to extend the tree to level $n+1$, and it suffices to show that the extension of any given
node $a_\nu$ (for $\nu \in \omega^n$) can be accomplished. But this amounts to repeating the argument for stage $0$
in the case where $A = A \cup \{a_\eta : \eta \subseteq \nu \}$. By assumption and Fact \ref{one-pt-extns}, 
this remains possible, so we continue.

Notice that the threat of all possible localizations is what makes continuation possible. That is, the schema
which says that ``$x$ is a 1-point extension of $A$'' simply says that $x$ remains (along with witnesses for $X_k$)
in each of an infinite set of localizations of $P_1$ with parameters from $A$. If this schema is inconsistent, there will be a
localization contradicting the hypothesis. 
\end{proof}

We can now characterize simplicity in terms of persistence:

\begin{theorem} \label{char-simple}
Let $\vp$ be a formula of $T$ and $\langle P_n \rangle$ its characteristic sequence. 

\begin{enumerate}
\item If the localization $\vp^f$ of $\vp$ is simple, then for each $P_\infty$-graph $A \subset P^f_1$ and for each $n<\omega$,
there exists a localization $P^{f_n}_1 \supset A$ of $P^f_1$ in which there is a uniform finite bound on the size of
a $P_n$-empty graph, i.e. there exists $m_n$ such that $X \subset P^f_1$ and $X^n \cap P_n = \emptyset$ implies $|X| \leq m_n$.
\item If localization $\vp^g$ of $\vp$ is not simple, then for all but finitely many $r<\omega$, $P^g_1$
contains an infinite $(r+1)$-empty graph.  
\end{enumerate}

\br
In other words, the following are equivalent for any positive base set $A$:
\begin{enumerate}
\item[(i)] There exists a localization $\vp^f$ of $\vp$ (with $\vp^f=\vp$ possible) such that $\vp^f$ is simple and $P^f_1 \supset A$.  
\item[(ii)] For each $n<\omega$, there exists a localization $P^{f_n}_1 \supset A$ in which there is a uniform finite bound
on the size of a $P_n$-empty graph.   
\end{enumerate}
\end{theorem}

\begin{proof}
It suffices to show the first two statements. 
(1) is Lemma \ref{simple-trees} applied to the formula $\vp^f$. (2) is the second clause of Observation \ref{trees-in-P1},
where ``almost all'' means for $r$ above $k$, the arity of dividing.
\end{proof}

\end{document}